\documentclass[final,3p,times,12pt]{elsarticle}
\usepackage[fleqn]{amsmath}
\usepackage{amsfonts}
\usepackage{amssymb}
\usepackage{graphicx}
\usepackage{subfigure}
\usepackage{color}
\usepackage{array}
\usepackage{mathrsfs}
\usepackage{float}
\usepackage{epstopdf}
\usepackage{booktabs}
\usepackage{tabularx}
\usepackage{lineno}
\usepackage{grffile}
\usepackage{amsthm}
\usepackage{extarrows}
\usepackage{amsmath}
\usepackage[colorlinks,linkcolor=blue,anchorcolor=blue,citecolor=blue,CJKbookmarks=True]{hyperref}
\biboptions{numbers,sort&compress}

\usepackage[american]{babel}
\usepackage{microtype}
\usepackage{ragged2e}
\allowdisplaybreaks[4]

\newcommand{\E}{\mathbb{E}}
\newcommand{\PP}{\mathbb{P}}
\newcommand{\RR}{\mathbb{R}}
\newcommand{\id}{\mathrm{id}}

\newcommand{\tabincell}[2]{\begin{tabular}{@{}#1@{}}#2\end{tabular}}
\def\[{\langle} \def\]{\rangle}
\newdefinition{lemma}{Lemma}[section]
\newdefinition{theorem}{Theorem}[section]
\newdefinition{definition}{Definition}[section]
\newdefinition{assumption}{Assumption}[section]
\newdefinition{proposition}{Proposition}[section]
\newdefinition{remark}{Remark}[section]
\newdefinition{claim}{Claim}[section]
\newdefinition{corollary}{Corollary}[section]
\newdefinition{example}{Example}[section]
\newdefinition{Appendix}{Appendix}[section]
\numberwithin{equation}{section}

\begin{document}
	\newpage
	\begin{frontmatter}
		\setcounter{page}{1}
		\title{Strong convergence rate of positivity-preserving truncated Euler--Maruyama method for multi-dimensional stochastic differential equations with positive solutions}
		\author{Xingwei Hu$^{a}$,~~Xinjie Dai$^b$,~~Aiguo Xiao$^{a,*}$}
		\cortext[cor1]{Corresponding author (Aiguo Xiao).\\
			\emph{Email addresses}: \texttt{xingweihu@smail.xtu.edu.cn} (X.\ Hu), 
			\texttt{dxj@ynu.edu.cn} (X.\ Dai),
			\texttt{xag@xtu.edu.cn} (A.\ Xiao). This research is supported by the National Natural Science Foundation of China (Nos.12471391, 12401547)}
		\address{$^a$School of Mathematics and Computational Science $\&$ Hunan Key Laboratory for Computation and Simulation in Science and Engineering, Xiangtan University, Xiangtan, Hunan 411105, China \\
			$^b$School of Mathematics and Statistics, Yunnan University, Kunming, Yunnan 650500, China}
		\date{}

		\begin{abstract}
			\par
			To construct positivity-preserving numerical methods, a vast majority of existing works employ transformation techniques such as the Lamperti transformation or logarithmic transformation. However, using these techniques often leads to the transformed stochastic differential equations (SDEs) not meeting the global monotonicity condition, particularly in multi-dimension case. This condition is essential for achieving strong convergence rates of numerical schemes. A pertinent question arises from this issue regarding the existence of an effective method with a convergence rate for solving multi-dimensional SDEs with positive solutions. This paper presents a positivity-preserving method that combines a novel truncated mapping with a truncated Euler--Maruyama discretization. We investigate both the strong convergence of the numerical method under some reasonable conditions. Furthermore, we demonstrate that this method achieves the optimal strong convergence order of 1/2 under certain additional assumptions. Numerical experiments are conducted to validate these theoretical results and demonstrate the positivity of the numerical solutions.
		\end{abstract}
		
	\end{frontmatter}
	
	\section{Introduction}
	\label{sec.1}
	Stochastic differential equations are crucial applications for mathematical modeling of random phenomena. Numerical methods for solving SDEs have been extensively applied to generate numerical solutions as substitutes for unavailable analytical solutions. However, many numerical methods, which usually fail to preserve the positivity of the original SDE, are not suitable for solving SDEs with positive solutions. A typical example is the explicit Euler--Maruyama (EM) method \cite{CM08}.\\
	\indent Recently, many positivity-preserving numerical methods have proposed. On the one hand, when considering specific models of SDEs with positive solutions, the construction of numerical methods and their analysis are finely recorded in references, which include the stochastic Lotka--Volterra competition model \cite{APPL2023, MF21}, the stochastic Lotka--Volterra predator-prey model \cite{HJZ2022}, the A\"it--Sahalia--type interest rate model \cite{DFM2023, AST, J2025, LCW2024, ASt}, the stochastic susceptible-infected-susceptible epidemic model \cite{AHOP2023, FOSC2021}, the Heston 3/2 model \cite{WG2025}, and so on. \\
	\indent On the other hand, for the scalar SDE which requires positive numerical approximations, scholars usually combine the Lamperti or logarithmic transformations with truncated/tamed/projected strategies \cite{pj2016,pj2017,GLMY2017,Hut02,Li2019,M15,M16,SS12,Eavc2016,MVZ2013,Gan01} and explicit/implicit discretization approaches to construct the positivity-preserving numerical methods \cite{HXW, LG2023, LNX2025, AL2014, TX2024, PPlogTM}. In \cite{DFF2024}, authors used truncated techniques to approximate the SDE without transformations, developing two positivity-preserving schemes: the positivity-preserving truncated EM and Milstein methods, with strong convergence rates of 1/2 and 1 respectively. However, they are inadequate for multi-dimensional SDEs valued in $\RR^d_+$. \\
	\indent For multi-dimensional stochastic Kolmogorov differential equations (SKDEs) with superlinear coefficients, \cite{YQX2024} proposed a positivity-preserving truncated EM method via the Lotka--Volterra system, achieving strong convergence. Additionally, \cite{YXF2024} presented an exponential EM scheme with strong convergence order arbitrarily close to 1/2. These methods perform well for the stochastic Lotka--Volterra system and even multi-dimensional SKDEs (a class of multi-dimensional SDEs). Recently, \cite{HCCC2025} proposed a Lagrange multiplier approach for constructing positive preserving scheme for stochastic complex systems, where the drift coefficients satisfy a one-sided Lipschitz condition, and the diffusion coefficients are globally Lipschitz. Furthermore, \cite{BCU1,BCU2} proposed positivity-preserving numerical methods for solving a class of stochastic heat equations.\\
	\indent This paper considers general multi-dimensional SDEs with non-global Lipschitz coefficients, whose solutions take values in $\RR_+^d$. We firstly demonstrate that the underlying equation has a unique strong solution $\{X(t)\}_{\{t\geq 0\}}$, with its sample path contained in $\mathbb{R}_+^d$, then present a novel positivity-preserving numerical method. Our approach utilizes a novel truncation technique which differs from the traditional one described in \cite{Li2019, M15}. Specifically, in Section 3, we define the truncation mapping $\pi_\Delta:\RR^d\rightarrow\RR_+^d$ to ensure each component of the vector $x$ being constrained within the range \(((\phi^{-1}(h(\Delta)))^{-1}, \phi^{-1}(h(\Delta)))\). In contrast, the standard truncation mapping $\pi_\Delta:\RR^d\rightarrow\RR^d$ only ensures the Euclidean norm $|x|\leq\phi^{-1}(h(\Delta))$, where functions $\phi$ and $h$ are selected based on the coefficients of the equation.\\
	\indent A main idea is to use the truncation mapping to ensure positivity of numerical approximations. Setting $\tilde{X}_0=X_0$, we compute
	\begin{flalign*}
		\tilde{X}_{k+1}=X_k+f(X_k)\Delta+g(X_k)\Delta B_k,
	\end{flalign*}
	where $\Delta=\frac{T}{N}$ for $N\in \mathbb{N}^+$. At each time step, applying the truncation mapping as $X_k=\pi_\Delta(\tilde{X}_k)$ yields the positivity-preserving sequence $\{X_k\}_{k=0}^N$, leading us to name the proposed method the positivity-preserving truncated EM (PPTEM) method. Under certain assumptions, we derive moment and inverse moment estimates for the numerical solutions. Furthermore, we utilize the inequality
	\begin{flalign*}
		\PP(\tilde{\xi}_\phi\leq T)\leq \frac{C}{(\phi^{-1}(h(\Delta)))^{\bar{p}\wedge\bar{q}}},
	\end{flalign*}
	where $\bar{p}$ and $\bar{q}$ come from Assumption 2.2, together with the stopping time
	\begin{flalign*}
		\tilde{\xi}_\phi:=\inf\{t_k\in[0,T]:\tilde{X}_{i}(t)\notin(\frac{1}{\phi^{-1}(h(\Delta))},\phi^{-1}(h(\Delta)))\hspace{0.4em}\text{for some}\hspace{0.4em}i=1,2,\cdots,d\}, 
	\end{flalign*}
	instead of the truncation mapping $X(t)=\pi_\Delta(X(t))$ when $\frac{1}{R}\leq X(t)\leq R$ used in \cite{DFF2024}. This new truncation mapping is universal, applicable to both one-dimensional and multi-dimensional cases. By combining the stochastically C-stable and B-consistent frameworks, we establish the optimal strong convergence rate of order 1/2. \\
	\indent The main contributions of this paper are as follows:
	\begin{itemize}
		\item For the general multi-dimensional SDEs with positive solutions, we demonstrate that the underlying equation admits a unique strong solution $\{X(t)\}_{\{t\geq 0\}}$, whose sample path is contained in $\RR_+^d$.
		\item Unlike the traditional truncated mapping, we propose a novel truncated mapping $\pi_\Delta$ (see \eqref{tm}) to ensure each component of the numerical solution vector is constrained within a bounded interval in $\RR_+$, which is a key technique to guarantee the positivity of numerical solutions.
		\item Compared with the existing methods \cite{YQX2024,YXF2024,HCCC2025} (a strong convergence rate of order arbitrarily close to 1/2), our PPTEM method exhibits the optimal strong convergence order 1/2 for general multi-dimensional SDEs with positive solutions and non-global Lipschitz coefficients.

	\end{itemize}
	
	\indent This paper is organized as follows. In Section 2, we introduce notations and assumptions while establishing useful lemmas and a corollary about the analytic solutions. In Section 3, we propose the PPTEM method and obtain useful lemmas and a corollary about the numerical solutions. In Section 4, we provide a detailed analysis of strong convergence and then demonstrate a strong convergence rate under additional assumptions. In Section 5, several numerical experiments illustrate our theoretical conclusions. Finally, we make a brief conclusion in Section 6.

	\section{Preliminaries and useful lemmas}
	\label{sec.2}
	
	\par Throughout this paper, we provide the following notations. \\
	\indent Let $\mathbb{N}^+$ denote the set of all positive integers. Let $A^T$ denote the transpose of the vector or matrix $A$. The symbols $|\cdot|$ and $\langle\cdot,\cdot\rangle$ denote the Euclidean norm and inner product of vectors in $\RR^d$, respectively. For a matrix $B$, its trace norm is defined as $|B|:=\sqrt{\text{trace}(B^TB)}$. 
	Let $\mathbb{E}$ denote the expectation under probability measure $\PP$. For 
	$u\in \mathbb{N}^+$, $L^u(\Omega,\RR^{d\times m})$ represents the family of  $\RR^{d\times m}$-valued random variables with norm $\Vert\xi\Vert_{L^u(\Omega;\RR^d)}=(\E[|\xi|^u])^\frac{1}{u}<\infty$. Let $\RR_+^d$ denote the positive cone in $\RR^d$, i.e., $\RR_+^d=\{x=(x_1,x_2,\cdots,x_d)^T\in\RR^d:x_i>0, \, \forall\, 1\leq i\leq d\}$.  For a set $A$, its indicator function is $I_A$, where $I_A(x)=1$ if $x\in A$ and $0$ otherwise. For two real numbers $a$ and $b$, set  $a \vee b=\max\left\{a,b\right\}$ and $a \wedge b=\min\left\{a,b\right\}$. The floor function of a real number $y$ is denoted by $\lfloor y \rfloor$, defined as the largest integer not exceeding $y$. We use $C$ to denote a generic constant, whose values may vary in different places. 
	\par Consider a $d$-dimensional SDE
	\begin{equation}\label{sde}
		\mathrm{d}X(t)=f(X(t))\mathrm{d}t+g(X(t))\mathrm{d}B(t),\quad t\in(0,T], \quad X(0)=X_0\in \RR^d_+,
	\end{equation}
	where $f=(f^1,f^2,\cdots,f^d)^T:\RR^d\rightarrow \RR^d$ and $g=(g^{i,j})_{d\times m}=(g^1,g^2,\cdots,g^m)=(g^T_1,g^T_2,\cdots,g^T_d)^T:\RR^d\rightarrow \RR^{d\times m}$. In this notation, $f^i:\RR^d\rightarrow\RR$ and $g^j:\RR^d\rightarrow\RR^d$ represent real-valued and $d\times 1$ vector functions, respectively, while $g_i:\RR^d\rightarrow \RR^m$ is a $1\times m$ vector function. $B(t)$ is an $m$-dimensional Brownian motion defined on a complete probability space $(\Omega, \mathcal{F}, \PP)$, equipped with a filtration $\left\{\mathcal{F}_t \right\}_{t \ge 0}$ that satisfies the usual conditions. 
	\begin{assumption}\label{as.1}
		Assume that the coefficients $f$ and $g$ satisfy the non-global Lipschitz condition: there exist constants $K_1>0$, $\alpha>0$ and $\beta>0$ such that for all $x, y\in \RR^d_+$,
		\begin{flalign*}
			|f(x)-f(y)|\vee|g(x)-g(y)|\leq K_1(1+|x|^\alpha+|y|^\alpha+|x|^{-\beta}+|y|^{-\beta})|x-y|.
		\end{flalign*} 
	\end{assumption}
	
	\begin{assumption}\label{as.2}
		Assume that there exist positive constants $\bar{x}_i>0,\bar{p}>1,\bar{q}>0$ and $K_{2,i}>0$ such that for any $x=(x_1,x_2,\cdots,x_d)^T\in\RR^d_+$ and any $i\in\{1,2,\cdots,d\}$,
		\begin{flalign*}
			\left\{ 
			\begin{aligned}
				&x_if^i(x)-\frac{\bar{q}+1}{2}|g_i(x)|^2\geq0, &x_i&\in(0,\bar{x}_i),\\
				&x_if^i(x)+\frac{\bar{p}-1}{2}|g_i(x)|^2\leq K_{2,i}(1+x_i^2), &x_i&\in[\bar{x}_i,\infty).
			\end{aligned}
			\right.
		\end{flalign*}
	\end{assumption}
	\par The hypotheses on coefficients ensure that the multi-dimensional SDE \eqref{sde} possesses unique global solution $\{X(t)\}_{t\in[0,T]}$ taking values in $\RR^d_+$, see Lemma \ref{Lm2.1}. Moreover, the A\"it--Sahalia model, the transformed CEV model, the stochastic Lokta--Volterra system, the SIRS epidemic model and the stochastic HIV/AIDs model satisfy the above assumptions (see Section 5).
	\begin{remark}\label{Remark1}
		One can derive from Assumption \ref{as.1} that 
		\begin{flalign*}
			|f(x)|\leq&|f(x)-f(1)|+|f(1)|\\
			\leq& K_1(1+|x|^{\alpha}+|x|^{-\beta}+|1|^{\alpha}+|1|^{-\beta})|x-1|+|f(1)|\\
			\leq& K_1(1+|x|^{\alpha}+|x|^{-\beta}+d^{\frac{\alpha}{2}}+d^{-\frac{\beta}{2}})(|x|+|1|)+|f(1)|\\
			\leq& 3d^{\frac{\alpha}{2}}K_1(1+|x|^{\alpha}+|x|^{-\beta})(|x|+d^{\frac{1}{2}})+|f(1)|\\
			\leq& 3d^{\frac{\alpha+1}{2}}K_1(1+|x|+|x|^{\alpha}+|x|^{\alpha+1}+|x|^{-\beta}+|x|^{-\beta+1})+|f(1)|.
		\end{flalign*}
		Then we conclude
		\begin{flalign*}
			|f(x)|\leq \left\{ 
			\begin{aligned}
				&9d^{\frac{\alpha+1}{2}}K_1(1+|x|^{\alpha+1})+|f(1)|, &|x|&\geq1,\\
				&12d^{\frac{\alpha+1}{2}}K_1(1+|x|^{-\beta})+|f(1)|, &|x|&<1.
			\end{aligned}
			\right.
		\end{flalign*}
		Therefore, Assumption \ref{as.1} means that
		\begin{flalign*}
			|f(x)|\vee|g(x)|\leq (21d^{\frac{\alpha+1}{2}}K_1\vee|f(1)|)(1+|x|^{\alpha+1}+|x|^{-\beta}),\quad\forall x\in\RR^d_+.
		\end{flalign*}
	\end{remark}
	
	\par In \cite[Lemma 2.1]{TX2024}, it was proven that the scalar SDE admits a unique positive solution by utilizing the one-dimensional versions of Assumptions \ref{as.1} and \ref{as.2}. Herein, we will show that the SDE \eqref{sde} has a unique global solution $\{X(t)\}_{t\in[0,T]}$ taking values in $\RR^d_+$, along with establishing the moment bound and the inverse moment of the exact solution.
	\begin{lemma}\label{Lm2.1}
		Let Assumptions \ref{as.1} and \ref{as.2} hold with $\bar{p}\geq2(\alpha+1)$ and $\bar{q}\geq2\beta$. Then the SDE \eqref{sde} has unique strong solution $\{X(t)\}_{t\in[0,T]}$ with satisfying the following equality
		\begin{flalign*}
			\PP(X(t)\in\RR^d_+, \forall t\in[0,T])=1.
		\end{flalign*}
		Furthermore, there exists a constant $C$ such that 
		\begin{flalign*}
			\PP(\gamma_n\leq T)\leq \frac{C}{n^{\bar{p}\wedge\bar{q}}},
		\end{flalign*}
		where $\gamma_n=\inf\{t\in[0,T]:X_{i}(t)\notin(n^{-1},n) \hspace{0.4em}\text{for some}\hspace{0.4em}i=1,2,\cdots,d\}$, $X_{i}(t)$ denotes the $i$-th element of $X(t)$. Besides, there exists a constant $C$ such that 
		\begin{flalign*}
			\sup_{ t \in[0,T]}\E[|X(t)|^{\bar{p}}]\leq C \quad \text{and} \quad \sup_{ t \in[0,T]}\E[|X(t)|^{-\bar{q}}]\leq C.
		\end{flalign*}
	\end{lemma}
	\begin{proof}
		Let $r$ be a positive integer. For any $X=(X_1,X_2,\cdots,X_d)^T\in \RR^d_+$, we set
		\begin{flalign*}
			\theta_r(X)=\big((r^{-1}\vee X_1)\wedge r,(r^{-1}\vee X_2)\wedge r,\cdots,(r^{-1}\vee X_d)\wedge r\big)^T.
		\end{flalign*} 
		As stated in Remark \ref{Remark1}, the functions
		\begin{flalign*}
			f_r(X) = f(\theta_r(X)) \quad \text{and} \quad g_r(X) = g(\theta_r(X))
		\end{flalign*}
		are globally Lipschitz continuous and linearly growing. By Theorem 2.3.4 in \cite{M01}, the solution to
		\begin{flalign*}
			\mathrm{d}X^r(t)=f_r(X^r(t))\mathrm{d}t+g_r(X^r(t))\mathrm{d}B(t)
		\end{flalign*}
		exists uniquely on $[0,T]$. Next, we define the stopping time $\tau_r=\inf\{t\in[0,T]:X^r_{i}(t)\notin(r^{-1},r) \hspace{0.4em}\text{for some}\hspace{0.4em}i=1,2,\cdots,d\}$, where $X^r_{i}(t)$ denotes the $i$-th element of $X^r(t)$. The uniqueness of $X^r(t)$ means that $X^m(t)=X^n(t)$ for $t\in[0,T\wedge\tau_n]$, where $m>n$. Therefore, $\tau_n$ is nondecreasing, and we set $\tau_\infty=\lim_{n\rightarrow\infty}\tau_n$.\\
		\indent Let $\omega\in\Omega$. For an arbitrary $t<\tau_\infty(\omega)$, there exists $r>0$ such that $t<\tau_r(\omega)\leq\tau_\infty(\omega)$. 
		For a fixed positive integer $n$ and arbitrary $t\in[0,T]$, we have
		\begin{flalign*}
			X(t\wedge\tau_n)&=X^n(t\wedge\tau_n)\\
			&=X_0+\int_{0}^{t\wedge\tau_n}f_n(X^n(s))\mathrm{d}s+\int_{0}^{t\wedge\tau_n}g_n(X^n(s))\mathrm{d}B(s)\\
			&=X_0+\int_{0}^{t\wedge\tau_n}f(X(s))\mathrm{d}s+\int_{0}^{t\wedge\tau_n}g(X(s))\mathrm{d}B(s).
		\end{flalign*}
		Define the Lyapunov function 
		\begin{flalign*}
			V(X)=\sum_{i=1}^{d}(X_i^{\bar{p}}+X_i^{-\bar{q}}),
		\end{flalign*}
		and its operator 
		\begin{flalign*}
			\mathcal{L}V(X)=V'(X)f(X)+\frac{1}{2}\text{trace}[g^T(X)V''(X)g(X)],
		\end{flalign*}
		where $V'(X)$ and $V''(X)$ denote the first and second derivative of $V(X)$, respectively. This follows that
		\begin{flalign*}
			\mathcal{L}V(X)=&V'(X)f(X)+\frac{1}{2}\text{trace}[g^T(X)V''(X)g(X)]\\
			=&\sum_{i=1}^{d}\Big(\bar{p}X_i^{\bar{p}-1}f^i(X)-\bar{q}X_i^{-\bar{q}-1}f^i(X)+\frac{1}{2}\big(\bar{p}(\bar{p}-1)X_i^{\bar{p}-2}+\frac{1}{2}\bar{q}(\bar{q}+1)X_i^{-\bar{q}-2})(\sum_{j=1}^{m}g^{i,j}(X))^2\Big)\\
			=&\sum_{i=1}^{d}\bar{p}X_i^{\bar{p}-2}[X_if^i(X)+\frac{\bar{p}-1}{2}|g_i(X)|^2]-\sum_{i=1}^{d}\bar{q}X_i^{\bar{q}-2}[X_if^i(X)-\frac{\bar{q}+1}{2}|g_i(X)|^2].
		\end{flalign*}
		Using Assumption \ref{as.2} leads to
		\begin{flalign*}
			\mathcal{L}V(X)\leq&\bar{p}\sum_{i=1}^{d}X_i^{\bar{p}-2}\Big(K_{2,i}(1+X_i^2)I_{\{X_i\geq\bar{x}_i\}}+(X_if^i(X)+\frac{\bar{p}-1}{2}|g_i(X)|^2)I_{\{0<X_i<\bar{x}_i\}}\Big)\\
			&-\bar{q}\sum_{i=1}^{d}X_i^{-\bar{q}-2}(X_if^i(X)-\frac{\bar{q}+1}{2}|g_i(X)|^2)I_{\{X_i\geq\bar{x}_i\}}\\
			\leq&C\Big(1+\sum_{i=1}^{d}X_i^{\bar{p}}I_{\{X_i\geq\bar{x}_i\}}\Big)+\bar{p}\sum_{i=1}^{d}X_i^{\bar{p}-2}(X_i|f^i(X)|+\frac{\bar{p}-1}{2}|g_i(X)|^2)I_{\{0<X_i<\bar{x}_i\}}\\
			&+\bar{q}\sum_{i=1}^{d}X_i^{-\bar{q}-2}(X_i|f^i(X)|+\frac{\bar{q}+1}{2}|g_i(X)|^2)I_{\{X_i\geq\bar{x}_i\}}\\
			\leq&C(1+\sum_{i=1}^{d}X_i^{\bar{p}})+C(|X||f(X)|+|g(X)|^2)\\
			\leq&C(1+\sum_{i=1}^{d}X_i^{\bar{p}}+|X|^{2\alpha+2}+|X|^{-2\beta}).
		\end{flalign*}
		It holds that
		\begin{flalign*}
			|X|^2\leq d\max_{1\leq i\leq d}X_i^2,
		\end{flalign*}
		so we have 
		\begin{flalign*}
			|X|^{2\alpha+2}\leq d^{\alpha+1}\max_{1\leq i\leq d} X_i^{2\alpha+2}\leq C\sum_{i=1}^{d}X_i^{2\alpha+2}.
		\end{flalign*}
		On the other hand, 
		\begin{flalign*}
			|X|^{-2\beta}=(\sum_{i=1}^{d}X_i^2)^{-\beta}\leq (d\min_{1\leq i \leq d}X_i^2)^{-\beta}\leq C\sum_{i=1}^{d}X_i^{-2\beta}.
		\end{flalign*}
		Thus, we get
		\begin{flalign*}
			\mathcal{L}V(X)\leq& C(1+\sum_{i=1}^{d}X_i^{\bar{p}}+\sum_{i=1}^{d}X_i^{2\alpha+2}+\sum_{i=1}^{d}X_i^{-2\beta})\\
			\leq& C(1+\sum_{i=1}^{d}X_i^{\bar{p}}+\sum_{i=1}^{d}X_i^{-\bar{q}})\\
			=&C(1+V(X)),
		\end{flalign*}
		where $\bar{p}\geq 2\alpha+2$ and $\bar{q}\geq 2\beta$ were used. Then $V(X)$ is an It\^o process with the stochastic differential equation given by 
		\begin{flalign*}
			\mathrm{d}V(X)=\mathcal{L}V(X)\mathrm{d}t+V'g(X)\mathrm{d}B(t).
		\end{flalign*}
		Dynkin's formula implies that there exists a constant $C$ such that 
		\begin{flalign}\label{eq21}
			\E V(X(\gamma_n\wedge T))\leq V(X(0))+CT+C\E\int_{0}^{\gamma_n\wedge T}V(X(t))\mathrm{d}t,
		\end{flalign}
		where $\gamma_n=\inf\{t\in[0,T]:X_{i}(t)\notin(n^{-1},n) \hspace{0.4em}\text{for some}\hspace{0.4em}i=1,2,\cdots,d\}$, $X_{i}(t)$ denotes the $i$-th element of $X(t)$. Then the Gr\"onwall inequality implies that
		\begin{flalign*}
			\E V(X(\gamma_n\wedge T))\leq Ce^{CT}\leq C.
		\end{flalign*}
		For every $\omega\in\{\gamma_n\leq T\}$, there is some $i$ such that $X_i(\tau_n,\omega)$ equals either $n$ or $\frac{1}{n}$. Hence 
		\begin{flalign*}
			\PP(\gamma_n\leq T)n^{\bar{p}\wedge\bar{q}}\leq \E [I_{\gamma_n\leq T}(\omega)V(X(\gamma_n,\omega))]=\E V(X(\gamma_n\leq T))\leq C.
		\end{flalign*}
		Therefore, we have 
		\begin{flalign*}
			\PP(\gamma_n\leq T)\leq \frac{C}{n^{\bar{p}\wedge\bar{q}}}
		\end{flalign*}
		and $\PP(\gamma_\infty>T)=1$. It means that the SDE \eqref{sde} has unique strong solution on [0, T] and
		\begin{flalign*}
			\PP(X(t)\in\RR_+^d,\forall t\in[0,T])=1.
		\end{flalign*} 
		Similarly, we have $\PP(\gamma_\infty>t)=1$. By the Fatou lemma, we have
		\begin{flalign*}
			\E V(X(t))\leq\mathop{\underline{\lim}}\limits_{n\rightarrow\infty}\E V(X(\gamma_n\wedge t))\leq C.
		\end{flalign*}
		Consequently, the inequality $\sum_{i=1}^{d}(X_i^{\bar{p}}+X_i^{-\bar{q}})\leq C$ holds. That is to say $\sum_{i=1}^{d}X_i^{-\bar{q}}\leq C$. Therefore, there exists $\delta\geq \delta_0>0$ such that $X_i>\delta$. Then for any $i$, we have $|X|>\sqrt{d}\delta$. Thus we have $|X|^{-\bar{q}}<(\sqrt{d}\delta)^{-\bar{q}}=d^{-\frac{\bar{q}}{2}}\delta^{-\bar{q}}\leq C$. Besides, $|X|^p=(\sum_{i=1}^{d}X_i^2)^{\frac{\bar{p}}{2}}\leq C(\sum_{i=1}^{d}X_i^{\bar{p}})\leq C$, where the H\"older inequality was used. Therefore, we have 
		\begin{flalign*}
			\E[|X(t)|^{\bar{p}}+|X(t)|^{-\bar{q}}]\leq C.
		\end{flalign*}
		The proof is completed.
	\end{proof}
	
	\section{Positivity-preserving truncated Euler--Maruyama method}
	\label{sec.3}
	
	\par Now we construct the PPTEM method. 
	As derived from \ref{Remark1}, we have
	\begin{flalign*}
		\frac{|f(x)|}{1+|x|} \vee\frac{ |g(x)| }{1+|x|}\leq H_0(1+|x|^\alpha+|x|^{-\beta-1}),
	\end{flalign*} 
	where $H_0\geq(21d^{\frac{\alpha+1}{2}}K_1\vee|f(1)|)$.
	Firstly, let $\phi(R)=2H_0R^{\alpha\vee(\beta+1)}$ be a strictly increasing function satisfying
	\begin{equation*}
		\sup_{R^{-1}\leq |x|\leq R} \frac{|f(x)|}{1+|x|} \vee\frac{ |g(x)| }{1+|x|} \leq \phi(R), \quad \forall R > 1.
	\end{equation*}
	Let $\phi^{-1}$ denote the inverse function of $\phi$, which is also a strictly increasing continuous function with $\phi^{-1}: (\phi(1), \infty) \rightarrow (1,\infty)$. Next, we select positive constants $\hat{U}$, $\hat{K_0}\geq 1$, $0<\bar{k}\leq\frac{1}{2}$, and a strictly decreasing function $h:(0, 1) \rightarrow (\phi(1), \infty)$ such that for all $\Delta\in(0,1)$,
	\begin{equation}\label{hD}
		\lim_{\Delta \rightarrow 0}h(\Delta) =\infty \quad \mbox{and} \quad \phi(\hat{K_0})\Delta^{-\bar{k}+\frac{1}{2}}\leq \Delta^{\frac{1}{2}}h(\Delta) \leq \hat{U}. 
	\end{equation}
	For a fixed $\Delta \in (0,1)$ and any $x \in \RR^d$, we define the truncated mapping $\pi_\Delta:\RR^d\rightarrow\RR_+^d$ by
	\begin{equation}\label{tm}
		\pi_\Delta(x) = \Big(\big((\phi^{-1}(h(\Delta)))^{-1}\vee x_1\big)\wedge \phi^{-1}(h(\Delta)),\cdots,\big((\phi^{-1}(h(\Delta)))^{-1}\vee x_d\big)\wedge \phi^{-1}(h(\Delta))\Big)^T.
	\end{equation}
	
	\begin{remark}\label{Remark2}
		By Remark \ref{Remark1} and $\phi(R)=2H_0R^{\alpha\vee(\beta+1)}$, if $\alpha\geq\beta+1$, we might as well let $\phi(R)=2H_0R^{\alpha}$, then  $\phi^{-1}(R)=(2H_0)^{-\frac{1}{\alpha}}R^{\frac{1}{\alpha}}$ holds accordingly. Thus, we observe from \eqref{hD} that 
		\begin{flalign*}
			\hat{K_0}\Delta^{-\frac{\bar{k}}{\alpha}}\leq\phi^{-1}(h(\Delta))\leq (2H_0)^{-\frac{1}{\alpha}}\hat{U}^{\frac{1}{\alpha}}\Delta^{-\frac{1}{2\alpha}}.
		\end{flalign*}
		Otherwise, $\hat{K_0}\Delta^{-\frac{\bar{k}}{\beta+1}}\leq\phi^{-1}(h(\Delta))\leq (2H_0)^{-\frac{1}{\beta+1}}\hat{U}^{\frac{1}{\beta+1}}\Delta^{-\frac{1}{2(\beta+1)}}$. 
	\end{remark}
	\par To introduce our numerical method, we define a uniform mesh $\mathcal{T}_N: 0=t_0<t_1<\cdots<t_N=T$ with $t_k=k\Delta$, where $\Delta=\frac{T}{N}$ for $N\in \mathbb{N}^+$. For any given step size $\Delta\in (0,1)$, we introduce the following truncated EM discrete approximation for \eqref{sde}, i.e.,
	\begin{flalign}\label{TEM}
		\left\{ 
		\begin{aligned}
			\tilde{X}_0&=X_0,  \\
			\tilde{X}_{k+1}&=X_k+f(X_k)\Delta+g(X_k)\Delta B_k,\\
			X_{k+1}&=\pi_\Delta(\tilde{X}_{k+1}), \quad k\in\{0,1,2,\cdots,N-1\}, 
		\end{aligned}
		\right.
	\end{flalign}
	where $\Delta B_k=B(t_{k+1})-B(t_k)$. Moreover, we define
	\begin{flalign*}
		\tilde{X}(t):=\tilde{X}_k \quad \text{and} \quad X_\Delta(t):=X_k,\quad t\in[t_k,t_{k+1}),k\in\{0,1,\cdots,N-1\}.
	\end{flalign*} 
	
	\indent Notice that $\{\tilde{X}(t)\}_{ t \in[0,T]}$ does not preserve positivity, whereas  $\{X_\Delta(t)\}_{ t \in[0,T]}$ does. We refer to the numerical scheme $\{X_\Delta(t)\}_{ t \in[0,T]}$ and \eqref{TEM} as the PPTEM method. An important step in getting the convergence results of the PPTEM method is to show the moment and inverse moment bounds of $\tilde {X}(t)$ and $X_\Delta(t)$. 
	\begin{lemma}\label{pp numerical integral}
		Assume that the conditions of Lemma \ref{Lm2.1} hold. Let $R$ be a positive integer and define the stopping time $\tilde{\xi}_R:=\inf\{t\in[0,T]:\tilde{X}_{i}(t)\notin(\frac{1}{R},R)\hspace{0.4em}\text{for some}\hspace{0.4em}i=1,2,\cdots,d\}$, where $\tilde{X}_i(t)$ denotes the $i$-th element of $\tilde{X}(t)$. Let $\Delta^*\in(0,1)$ be sufficiently small such that $\phi^{-1}(h(\Delta^*))\geq R$. Then there exists a positive constant $C$ independent of $\Delta$ such that
		\begin{equation}\label{numerical1}
			\sup_{\Delta\in (0,\Delta^*]}\sup_{ t \in[0,T]}\E[|\tilde{X}(t)|^{\bar{p}}]\leq C \quad \text{and} \quad \sup_{\Delta\in(0,\Delta^*]}\sup_{ t \in[0,T]}\E[|\tilde{X}(t)|^{-\bar{q}}]\leq C.
		\end{equation}
	\end{lemma}
	\begin{proof}
		Firstly, we see that $\tilde{X}_{i}(s)\in(\frac{1}{R},R)$ for any $i\in\{1,2,\cdots,d\}$ and $s\in[T\wedge\tilde{\xi}_R]$. Since $\phi^{-1}(h(\Delta))\geq R$, we have $\tilde{X}_{i}(s)\in(\frac{1}{\phi^{-1}(h(\Delta))},\phi^{-1}(h(\Delta)))$ for any $i\in\{1,2,\cdots,d\}$ and hence $X_k=\tilde{X}_k$ for $t_k\in[T\wedge\tilde{\xi}_R]$. Then for an arbitrarily fixed $\Delta\in(0,1)$ and any $0\leq t\leq T$, there exists unique integer $k\geq0$ such that for $t_k\leq t<t_{k+1}$, it follows from \eqref{TEM} that
		\begin{flalign*}
			\tilde{X}(t\wedge\tilde{\xi}_R)=X_0+\int_{0}^{t_k\wedge\tilde{\xi}_R}f(\tilde{X}(s))\mathrm{d}s+\int_{0}^{t_k\wedge\tilde{\xi}_R}g(\tilde{X}(s))\mathrm{d}B(s).
		\end{flalign*}
		\indent Using the techniques in the proof of Lemma \ref{Lm2.1} leads to
		\begin{flalign*}
			\sup_{ t \in[0,T]}\E[|\tilde{X}(t)|^{\bar{p}}]\leq C \quad \text{and} \quad \sup_{ t \in[0,T]}\E[|\tilde{X}(t)|^{-\bar{q}}]\leq C, \quad \forall\Delta\in(0,\Delta^*],
		\end{flalign*}
		where $C$ is independent of $\Delta$. \hfill
	\end{proof}
	\begin{corollary}\label{Corollary2}
		Let the conditions of Lemma \ref{pp numerical integral} hold. Then there exists a positive constant $C$ independent of $\Delta$ such that
		\begin{flalign}\label{supsupED}
			\sup_{\Delta\in (0,\Delta^*]}\sup_{ t \in[0,T]}\E[|X_\Delta(t)|^{\bar{p}}]\leq C\quad \text{and}\quad \sup_{\Delta\in (0,\Delta^*]}\sup_{ t \in[0,T]}\E[|X_\Delta(t)|^{-\bar{q}}]\leq C.
		\end{flalign}
		Furthermore, define the stopping time $\tilde{\xi}_\phi:=\inf\{t_k\in[0,T]:\tilde{X}_{i}(t)\notin(\frac{1}{\phi^{-1}(h(\Delta))},\phi^{-1}(h(\Delta)))\hspace{0.4em}\text{for some}\hspace{0.4em}i=1,2,\cdots,d\}$. Then for any $\bar{p}\geq2(\alpha+1)$ and $\bar{q}\geq2\beta$, there exists a positive constant $C$ independent of $\Delta$ such that
		\begin{flalign}\label{corollary2}
			\PP(\tilde{\xi}_\phi\leq T)\leq \frac{C}{(\phi^{-1}(h(\Delta)))^{\bar{p}\wedge\bar{q}}}.
		\end{flalign}
	\end{corollary}
	\begin{proof}
		Following the proof of Lemma \ref{pp numerical integral}, one can derive \eqref{supsupED}. Moreover, \eqref{corollary2} can be obtained, and its proof is analogous to that of Lemma \ref{Lm2.1}.\hfill
	\end{proof}

	\section{Strong convergence analysis}
	\par From Lemma \ref{Lm2.1}, we again consider the SDE
	\begin{flalign}\label{Y(t)}
		\mathrm{d}X^r(t)=f_r(X^r(t))\mathrm{d}t+g_r(X^r(t))\mathrm{d}B(t)
	\end{flalign}
	for $t\geq 0$ with the initial value $X^r(0)=X(0)$. Recall that \eqref{Y(t)} has unique global solution $\{X^r(t)\}_{t\geq0}$.\\
	\indent Fixing the stepsize $\Delta\in(0,1)$ and applying the EM method to \eqref{Y(t)} yield that
	\begin{flalign}\label{Y_k}
		\left\{ 
		\begin{aligned}
			Y_0&=X(0),  \\
			Y_{k+1}&=Y_k+f_r(Y_k)\Delta+g_r(Y_k)\Delta B_k, \quad k \in\{0,1,\cdots,N-1\}.
		\end{aligned}
		\right.
	\end{flalign}
	Then we introduce the continuous-time version of \eqref{Y_k} as follows
	\begin{flalign*}
		Y(t)=Y_k,\quad \forall t\in[t_k,t_{k+1}),
	\end{flalign*}
	and the It\^o process
	\begin{flalign*}
		\tilde{Y}(t)=Y_0+\int_{0}^{t}f_r(Y(s))\mathrm{d}s+\int_{0}^{t}g_r(Y(s))\mathrm{d}B(s).
	\end{flalign*}
	It is well known that (see e.g., \cite{M01})
	\begin{flalign}\label{p_0}
		\E[\sup_{ t \in[0,T]}|\tilde{Y}(t)-X^r(t)|^{p_0}]\leq C_r\Delta^\frac{p_0}{2},
	\end{flalign}
	where $C_r$ is a positive constant dependent on $r$ and $T$ but independent of $\Delta$.\\
	\indent Notice that 
	\begin{flalign}\label{a.s.}
		X(t\wedge\tau_r)=X^r(t\wedge\tau_r)\quad a.s., \quad\forall t\geq 0.
	\end{flalign}
	\eqref{p_0} and \eqref{a.s.} indicate
	\begin{flalign}\label{tYX}
		\E\big[\sup_{ t \in[0,T]}|\tilde{Y}(t\wedge\tau_r)-X(t\wedge\tau_r)|^{p_0}\big]\leq C_r\Delta^\frac{p_0}{2}, \quad\forall\Delta\in(0,1).
	\end{flalign}
	
	\begin{lemma}\label{Lemma4.1}
		Let Assumptions \ref{as.1} and \ref{as.2} hold and $n\geq3$ be a sufficiently large integer such that
		\begin{flalign}\label{con3.3}
			\Big(\frac{2n}{2n-1}\Big)^p(T+1)^\frac{p}{2n}\leq2.
		\end{flalign}
		Then
		\begin{flalign}\label{as3.3}
			\E[\sup_{ t \in[0,T]}|\tilde{Y}(t)-Y(t)|^p]\leq C_r\Delta^{\frac{p(n-1)}{2n}}.
		\end{flalign}
	\end{lemma}
	\indent This lemma is proved in \ref{Appendix11}, which is crucial for the proof of Theorem \ref{Thm4.1}.

	\subsection{The p-th moment convergence}
	
	\begin{theorem}\label{Thm4.1}
		Let the conditions in Lemma \ref{Lm2.1} hold. Then for $p> 0$, the PPTEM solution $X_\Delta(t)$ is strongly convergent to the exact solution $X(t)$, i.e.,
		\begin{equation*}
			\lim_{\Delta\rightarrow 0}\E[|X(T)-X_\Delta(T)|^p]=0.
		\end{equation*}
	\end{theorem}
	\begin{proof}
		For an arbitrary $\epsilon>0$, we derive from Lemma \ref{Lm2.1} that there exists $r=r(\epsilon)$ sufficiently large such that $\PP(\gamma_r\leq T)\leq \frac{\epsilon}{2}$. Setting $\Omega_1 = \{\gamma_r>T\}=\{\frac{1}{r}<X_i(t)<r \hspace{0.2cm} \text{for}\hspace{0.2cm} \text{all} \hspace{0.2cm} 1\leq i\leq n\hspace{0.2cm} \text{and} \hspace{0.2cm}0\leq t\leq T\}$,
		we have 
		\begin{flalign}\label{Omega1}
			\PP(\Omega_1)\geq1-\frac{\epsilon}{2}.
		\end{flalign}
		With this $r$, we define $\Omega_2=\Big\{\sup_{ t \in[0,T]}|\tilde{Y}(t\wedge\gamma_r)-X(t\wedge\gamma_r)|<\frac{1}{2r}\Big\}$. By virtue of the Chebyshev inequality to \eqref{p_0}, we derive
		\begin{flalign*}
			\PP(\Omega_2^c)=\PP\Big(\sup_{ t \in[0,T]}|\tilde{Y}(t\wedge\gamma_r)-X(t\wedge\gamma_r)|\geq\frac{1}{2r}\Big)\leq C_r\Delta^\frac{p_0}{2},\quad \forall \Delta\in(0,1).
		\end{flalign*}
		Then one can derive that there exists $\Delta_1=\Delta_1(\epsilon)\in(0,1)$ sufficiently small such that
		\begin{flalign}\label{Omega2}
			\PP(\Omega_2)\geq 1-\frac{\epsilon}{2},\quad\forall\Delta\in(0,\Delta_1].
		\end{flalign}
		Let $\Omega_0=\Omega_1\cap\Omega_2$. Then it follows from \eqref{Omega1} and \eqref{Omega2} that
		\begin{flalign}\label{Omega0}
			\PP(\Omega_0)\geq 1-\epsilon, \quad\forall\Delta\in(0,\Delta_1].
		\end{flalign}
		For any $\omega\in\Omega_0$ and $\Delta\in(0,\Delta_1]$, we get
		\begin{flalign*}
			\sup_{ t \in[0,T]}Y_i(t)\leq \sup_{ t \in[0,T]}\tilde{Y}_i(t)\leq&\sup_{ t \in[0,T]}X_i(t)+\sup_{ t \in[0,T]}|\tilde{Y}_i(t)-X_i(t)|
			<r+\sup_{ t \in[0,T]}|\tilde{Y}(t)-X(t)|<r+\frac{1}{2r}
		\end{flalign*}
		and
		\begin{flalign*}
			\inf_{t \in[0,T]}Y_i(t)\geq\inf_{t \in[0,T]}\tilde{Y}_i(t)\geq&\inf_{ t \in[0,T]}X_i(t)-\sup_{ t \in[0,T]}|X_i(t)-\tilde{Y}_i(t)|\\
			>&\inf_{ t \in[0,T]}X_i(t)-\sup_{ t \in[0,T]}|X(t)-\tilde{Y}(t)|>\frac{1}{r}-\frac{1}{2r}=\frac{1}{2r}.
		\end{flalign*}
		Then there exists $\Delta_0\in(0,\Delta_1]$ sufficiently small such that $\phi^{-1}(h(\Delta_0))\geq 2r$, which means that 
		\begin{flalign*}
			\frac{1}{\phi^{-1}(h(\Delta_0))}<\inf_{ t \in[0,T]}Y_i(t)\leq\sup_{ t \in[0,T]}Y_i(t)<\phi^{-1}(h(\Delta_0)).
		\end{flalign*}
		Therefore, for any $\omega\in\Omega_0$ and $\Delta\in(0,\Delta_0]$, we have
		\begin{flalign}\label{YtX}
			Y(t)=\tilde{X}(t)=X_\Delta(t),\quad \forall t\in[0,T].
		\end{flalign}
		\indent Next we show the deserved assertion. For any $\Delta\in(0,\Delta_0]$ and $\bar{p}>p$,
		\begin{flalign}\label{maineq}
			&\E[|X_\Delta(T)-X(T)|^p]=\E[|X_\Delta(T)-X(T)|^pI_{\Omega_0}]+\E[|X_\Delta(T)-X(T)|^pI_{\Omega_0^c}]\nonumber\\
			\leq&\E\big[\sup_{ t \in[0,T]}|X_\Delta(t)-X(t)|^pI_{\Omega_0}\big]+\big(\PP(\Omega_0^c)\big)^{\frac{\bar{p}-p}{\bar{p}}}\big(\E[|X_\Delta(T)-X(T)|^{\bar{p}}]\big)^{\frac{p}{\bar{p}}}.
		\end{flalign}
		By using \eqref{tYX}, \eqref{YtX} and Lemma \ref{Lemma4.1}, we can derive
		\begin{flalign}\label{Esup}
			&\E\big[\sup_{ t \in[0,T]}|X_\Delta(t)-X(t)|^pI_{\Omega_0}\big]=\E\big[\sup_{ t \in[0,T]}|Y(t)-X(t)|^pI_{\Omega_0}\big]\nonumber\\
			\leq&\E\big[\sup_{ t \in[0,T]}|Y(t)-\tilde{Y}(t)|^p\big]+\E\big[\sup_{ t \in[0,T]}|\tilde{Y}(t)-X(t)|^pI_{\Omega_1}\big]\nonumber\\
			\leq&\E\big[\sup_{ t \in[0,T]}|Y(t)-\tilde{Y}(t)|^p\big]+\E\big[\sup_{ t \in[0,T]}|\tilde{Y}(t\wedge\gamma_r)-X(t\wedge\gamma_r)|^p\big]\nonumber\\
			\leq&C_r\Delta^\frac{p}{3},\quad \forall \Delta\in(0,\Delta_0].
		\end{flalign}
		Besides, Lemmas \ref{Lm2.1} and \ref{pp numerical integral} indicate
		\begin{flalign}\label{EEC}
			\big(\E[|X_\Delta(T)-X(T)|^{\bar{p}}]\big)^{\frac{p}{\bar{p}}}\leq 2^{\bar{p}}\big(\E[|X_\Delta(T)|^{\bar{p}}+|X(T)|^{\bar{p}}]\big)^{\frac{p}{\bar{p}}}\leq C.
		\end{flalign}
		Substituting \eqref{Omega0}, \eqref{Esup} and \eqref{EEC} to \eqref{maineq} leads to
		\begin{flalign*}
			\E[|X_\Delta(T)-X(T)|^p]\leq C_r\Delta^\frac{p}{3}+C\epsilon^{\frac{\bar{p}-p}{\bar{p}}},\quad\forall \Delta\in(0,\Delta_0].
		\end{flalign*}
		Finally,
		\begin{flalign*}
			\lim_{\Delta \rightarrow 0}\E[|X_\Delta(T)-X(T)|^p]\leq C\epsilon^{\frac{\bar{p}-p}{\bar{p}}}.
		\end{flalign*}
		By the arbitrariness of $\epsilon$, the desired assertion holds.\hfill
	\end{proof}
	
	\subsection{Strong convergence rate}
	\indent To give a strong convergence rate, we impose the following additional condition on $f$ and $g$. 
	\begin{assumption}\label{as.3}
		Assume that there exist positive constants $p>2$ and $K_3$ such that 
		\begin{flalign*}
			\langle x-y, f(x)-f(y) \rangle+\frac{p-1}{2}| g(x)-g(y) |^2\leq K_3| x-y|^2, \quad \forall x,y\in \RR^d_+.
		\end{flalign*}
	\end{assumption}
	\par Before presenting the result on the strong convergence rate, we first rewrite the truncated Euler-Maruyama (TEM) method $\Psi$ as a one-step scheme
	\begin{equation}\label{one-step}
		\Psi (X,\Delta) = \pi_\Delta(X) +f_\Delta(X)\Delta+g_\Delta(X)\Delta B_k, 
	\end{equation}
	which approximates $X(t+\Delta)$. Here, $X\in L^2(\Omega,\mathcal{F}_t,\PP)$, $\Psi (X,\Delta)\in L^2(\Omega,\mathcal{F}_{t+\Delta},\PP)$, and the truncated functions defined by
	\begin{equation*}
		f_\Delta(X)=f(\pi_\Delta(X)) \quad \text{and} \quad g_\Delta(X)=g(\pi_\Delta(X)).
	\end{equation*}
	By virtue of the stochastic C-stable and B-consistent framework, our ultimate goal is to establish an optimal strong convergence order 1/2. To this end, we outline the definitions of stochastic C-stability and B-consistency below (for further details, see \cite{pj2016} and \cite{pj2017}). By applying these techniques, we derive Proposition \ref{Pro 1}, which is critical for obtaining the strong convergence rate.
	\begin{definition}\label{Def1}(Definition 3.2 in \cite{pj2016})
		A stochastic one-step method $(\Psi,\Delta)$ is called stochastically C-stable if there exist a constant $C_{\mathrm{stab}}$ and a parameter value $\mu \in (1,\infty)$ such that for all $t \in [0,T]$ and all random variables $Y, Z \in L^2(\Omega,\mathcal{F}_t,\PP)$, it holds that
		\begin{flalign}\label{Cstab}
			&\big\| \E \big[ \Psi(Y,\Delta)-\Psi(Z,\Delta) | \mathcal{F}_{t} \big]\big\|_{L^2(\Omega;\RR^d)}^2+ \mu \big\| \big( \id - \E [ \, \cdot \, |\mathcal{F}_{t} ] \big) \big(\Psi(Y,\Delta)-\Psi(Z,\Delta) \big) \big\|^2_{L^2(\Omega;\RR^d)}\nonumber\\
			\leq& \big(1 + C_{\mathrm{stab}} \Delta \big)\big\| Y - Z \big\|_{L^2(\Omega;\RR^d)}^2.    
		\end{flalign}
	\end{definition}
	\begin{definition}(Definition 3.3 in \cite{pj2016}) \label{Def2}
		A stochastic one-step method $(\Psi,\Delta)$ is called stochastically B-consistent of order $\gamma_0 > 0$ to \eqref{sde} if there exists a constant $C_{\mathrm{cons}}$ such that for all $t \in [0,T]$, it holds that
		\begin{flalign}\label{bcon1}
			\begin{split}
				\big\| \E \big[ X(t+\Delta) - \Psi(X(t),\Delta) | \mathcal{F}_{t} \big] 
				\big\|_{L^2(\Omega;\RR^d)} \le C_{\mathrm{cons}} \Delta^{ \gamma_0 + 1}
			\end{split}
		\end{flalign}
		and
		\begin{align}\label{bcon2}
			\begin{split}
				\big\| \big( \id - \E [ \, \cdot \, | \mathcal{F}_{t} ] \big)\big( X(t + \Delta) - \Psi(X(t),\Delta) \big)\big\|_{L^2(\Omega;\RR^d)} \leq C_{\mathrm{cons}} \Delta^{ \gamma_0 +\frac{1}{2}},
			\end{split}
		\end{align}
		where $X(t)$ denotes the exact solution to \eqref{sde}.
	\end{definition}
	\par We will focus on the strong convergence rate of the PPTEM method. To proceed, we first establish the following lemmas for later use.  
	\begin{lemma}\label{Lemma4.2}
		Let Assumptions \ref{as.1} and \ref{as.3} hold. Then for all $\Delta\in(0,1)$ and $x,y\in \RR^d_+$, we have
		\begin{equation}\label{eq piD}
			|\pi_\Delta(x)-\pi_\Delta(y)+(f_\Delta(x)-f_\Delta(y))\Delta|^2+\mu\Delta|g_\Delta(x)-g_\Delta(y)|^2\leq(1+C\Delta)|x-y|^2.
		\end{equation}
	\end{lemma}
	
	\indent This lemma is proved in \ref{Appendix12}.
	
	\begin{lemma}\label{Lemma4.3}
		Let the conditions in Lemma \ref{Lm2.1} hold. Then for any $\Delta\in(0,1)$, any $p\geq2$, and any $0\leq t\leq T$, we have
		\begin{flalign}\label{eq ED}
			\E[|X(t)-X(t_k)|^p]\leq C\Delta^\frac{p}{2},
		\end{flalign}
		where $t_k=\lfloor \frac{t}{\Delta} \rfloor\Delta$.
	\end{lemma}
	
	\indent This lemma is proved in \ref{Appendix13}.
	
	\indent The next two theorems establish the results on stochastic C-stability and B-consistency for the TEM method, respectively.
	
	\begin{theorem}\label{Thm4.2}
		Let Assumptions \ref{as.1} and \ref{as.3} hold. Then for any step size $\Delta\in(0,1)$, the TEM method is stochastically C-stable. 
	\end{theorem}
	\begin{proof}
		Recalling Definition \ref{Def1}, for the TEM method \eqref{one-step}, we have 
		\begin{equation*}
			\E \big[ \Psi(Y,\Delta)-\Psi(Z,\Delta) | \mathcal{F}_{t} \big]=\pi_\Delta(Y)-\pi_\Delta(Z)+(f_\Delta(Y)-f_\Delta(Z))\Delta
		\end{equation*}
		and
		\begin{equation*}
			\big( \id - \E [ \, \cdot \, |\mathcal{F}_{t} ] \big) \big(\Psi(Y,\Delta)-\Psi(Z,\Delta) \big)=(g_\Delta(Y)-g_\Delta(Z))\Delta B_k.
		\end{equation*}
		Then from the It\^o isometry and Lemma \ref{Lemma4.2}, it follows that
		\begin{align*}
			\begin{split}
				&\big\|\pi_\Delta(Y)-\pi_\Delta(Z)+(f_\Delta(Y)-f_\Delta(Z))\Delta\big\|_{L^2(\Omega,\RR^d)}^2+ \mu \big\|  (g_\Delta(Y)-g_\Delta(Z))\Delta B_k \big\|^2_{L^2(\Omega,\RR^d)}\\
				=&\E\Big[|\pi_\Delta(Y)-\pi_\Delta(Z)+(f_\Delta(Y)-f_\Delta(Z))\Delta|^2+\mu\Delta|g_\Delta(Y)-g_\Delta(Z)|^2\Big]\\
				\leq&(1+C\Delta)\E[|Y-Z|^2]= (1 + C\Delta)\|Y - Z\|_{L^2(\Omega,\RR^d)}^2,  
			\end{split}
		\end{align*}
		which validates the equation \eqref{Cstab}. \hfill
	\end{proof}
	\begin{theorem}\label{Thm4.3}
		Let Assumptions \ref{as.1}, \ref{as.2}, and \ref{as.3} hold with $\alpha\vee(\beta+1)\leq \bar{p}+\bar{q}$. Then for any step size $\Delta\in(0,1)$, the TEM method is stochastically B-consistent of order $\gamma_0=\frac{1}{2}$. 
	\end{theorem}
	\begin{proof}
		It follows from \eqref{one-step} that
		\begin{flalign*}
			&X(t_k+\Delta) - \Psi(X(t_k),\Delta)\\
			=&X(t_k)-\pi_\Delta(X(t_k))+\int_{t_k}^{t_{k+1}}f(X(s))-f_\Delta(X(t_k))\mathrm{d}s+\int_{t_k}^{t_{k+1}}g(X(s))-g_\Delta(X(t_k))\mathrm{d}B(s)\\
			=&X(t_k)-\pi_\Delta(X(t_k))+\big(f(X(t_k))-f_\Delta(X(t_k))\big)\Delta+\big(g(X(t_k))-g_\Delta(X(t_k))\big)\Delta B_k\\
			&+\int_{t_k}^{t_{k+1}}f(X(s))-f(X(t_k))\mathrm{d}s+\int_{t_k}^{t_{k+1}}g(X(s))-g(X(t_k))\mathrm{d}B(s).
		\end{flalign*}
		Then we estimate \eqref{bcon1} as follows
		\begin{flalign}\label{bI}
			&\big\| \E \big[X(t_k+\Delta)-\Psi(X(t_k),\Delta)| \mathcal{F}_{t_k} \big] \big\|_{L^2(\Omega,\RR^d)}\nonumber\\ 
			\leq& \big\|X(t_k)-\pi_\Delta(X(t_k)) \big\|_{L^2(\Omega,\RR^d)}+\big\|f(X(t_k))-f_\Delta(X(t_k)) \big\|_{L^2(\Omega,\RR^d)}\Delta\nonumber\\
			&+\int_{t_k}^{t_{k+1}}\big\| \E \big[f(X(s))-f(X(t_k)) | \mathcal{F}_{t_k} \big] \big\|_{L^2(\Omega,\RR^d)}\mathrm{d}s\nonumber \\
			:=&\bar{I}_1+\bar{I}_2+\bar{I}_3.
		\end{flalign}
		\indent Let $\gamma_{\phi}:=\inf\{t_k\in[0,T]:X_i(t_k)\notin\big(\frac{1}{\phi^{-1}(h(\Delta))},\phi^{-1}(h(\Delta))\big)\hspace{0.4em}\text{for some}\hspace{0.4em}i=1,2,\cdots,d\}$. Then we decompose $\bar{I}_1$ as follows
		\begin{flalign*}
			&\E\big[|X(t_k)-\pi_\Delta(X(t_k))|^2\big]
			=\E\big[|X(t_k)-\pi_\Delta(X(t_k))|^2I_{\{\gamma_{\phi}\leq T\}}\big]\\
			\leq&\E\big[|X(t_k)|^2I_{\{\gamma_{\phi}\leq T\}}\big]+\E\big[|\pi_\Delta(X(t_k))|^2I_{\{\gamma_{\phi}\leq T\}}\big]\\
			:=&\bar{I}_{11}+\bar{I}_{12}.
		\end{flalign*}
		By Remark \ref{Remark2}, we have $(\phi^{-1}(h(\Delta)))^{-1}\leq C\Delta^{\frac{\bar{k}}{\alpha\vee(\beta+1)}}$. Applying the H$\rm \ddot{o}$lder inequality and Lemma \ref{Lm2.1} to $\bar{I}_{11}$ yields that
		\begin{flalign*}
			\E\big[|X(t_k)|^2I_{\{\tau_\phi\leq T\}}\big]
			\leq&\Big(\E[|X(t_k)|^{\bar{p}}]\Big)^\frac{2}{\bar{p}}\Big(\PP\big(\gamma_{\phi}\leq T\big)\Big)^\frac{\bar{p}-2}{\bar{p}}\leq C\Delta^{\frac{\bar{k}(\bar{p}-2)(\bar{p}\wedge\bar{q})}{\bar{p}(\alpha\vee(\beta+1))}}.
		\end{flalign*}
		Then $\bar{I}_{12}$ follows from Lemma \ref{Lm2.1} that
		\begin{flalign*}
			\E\big[|\pi_\Delta(X(t_k))|^2I_{\{\gamma_{\phi}\leq T\}}\big]\leq\sqrt{d}(\phi^{-1}(h(\Delta)))^2\frac{C}{(\phi^{-1}(h(\Delta)))^{\bar{p}\wedge\bar{q}}}\leq C\Delta^{\frac{\bar{k}(\bar{p}\wedge\bar{q}-2)}{\alpha\vee(\beta+1)}}.
		\end{flalign*} 
		Therefore,
		\begin{flalign*}
			\E\Big[|X(t_k)-\pi_\Delta(X(t_k))|^2\Big]\leq C\Delta^{\frac{\bar{k}(\bar{p}-2)(\bar{p}\wedge\bar{q})}{\bar{p}(\alpha\vee(\beta+1))}\wedge {\frac{\bar{k}(\bar{p}\wedge\bar{q}-2)}{\alpha\vee(\beta+1)}}}.
		\end{flalign*}
		\par Choosing $\bar{k}=\frac{1}{2}$ and $\bar{p}=\bar{q}=6(\alpha\vee(\beta+1))+2$, we get
		\begin{flalign*}
			\E\Big[|X(t_k)-\pi_\Delta(X(t_k))|^2\Big]\leq C\Delta^3.
		\end{flalign*}
		Consequently, 
		\begin{flalign}\label{bI1}
			\bar{I}_1=\big\|X(t_k)-\pi_\Delta(X(t_k)) \big\|_{L^2(\Omega,\RR^d)}=\Big(\E\Big[|X(t_k)-\pi_\Delta(X(t_k))|^2\Big]\Big)^\frac{1}{2}\leq C \Delta^\frac{3}{2}.
		\end{flalign}
		For $\bar{I}_2$, it follows from Assumption \ref{as.1} that 
		\begin{flalign*}
			&|f(X(t_k))-f_\Delta(X(t_k))|^2\\
			\leq&C\Big(1+|X(t_k)|^\alpha+|\pi_\Delta(
			X(t_k))|^\alpha+|X(t_k)|^{-\beta}+|\pi_\Delta(X(t_k))|^{-\beta}\Big)^2|X(t_k)-\pi_\Delta(X(t_k))|^2\\
			\leq& C\big(1+|X(t_k)|^{2\alpha}+|\pi_\Delta(X(t_k))|^{2\alpha}+|X(t_k)|^{-2\beta}+|\pi_\Delta(X(t_k))|^{-2\beta}\big)|X(t_k)-\pi_\Delta(X(t_k))|^2I_{\{\gamma_{\phi}\leq T\}}\\
			\leq&C\big(1+|X(t_k)|^{-2\beta}\big)(\phi^{-1}(h(\Delta)))^{-2}I_{\{\gamma_{\phi}\leq T\}}+C\big(1+|X(t_k)|^{2\alpha}\big)|X(t_k)|^2I_{\{\gamma_{\phi}\leq T\}}\\
			:=&\bar{I}_{21}+\bar{I}_{22}.
		\end{flalign*}
		For $\bar{I}_{21}$, taking the expectation and using the H$\rm \ddot{o}$lder inequality yield that 
		\begin{flalign*}
			&\E\Big[\big(1+|X(t_k)|^{-2\beta}\big)(\phi^{-1}(h(\Delta)))^{-2}I_{\{\gamma_{\phi}\leq T\}}\Big]\\
			=&\E\Big[(\phi^{-1}(h(\Delta)))^{-2}I_{\{\gamma_{\phi}\leq T\}}\Big]+\E\Big[|X(t_k)|^{-2\beta}(\phi^{-1}(h(\Delta)))^{-2}I_{\{\gamma_{\phi}\leq T\}}\Big]\\
			\leq&C\Delta^{\frac{\bar{k}(\bar{p}\wedge\bar{q}+2)}{\alpha\vee(\beta+1)}}+(\phi^{-1}(h(\Delta)))^{-2}\big( \E[|X(t_k)|^{-\bar{q}}]\big)^\frac{2\beta}{\bar{q}}\big(\PP(\gamma_{\phi}\leq T)\big)^\frac{\bar{q}-2\beta}{\bar{q}}\\
			\leq& 
			C\Delta^{\frac{\bar{k}(\bar{p}\wedge\bar{q}+2)}{\alpha\vee(\beta+1)}\wedge\frac{\bar{k}(\bar{q}-2\beta)(\bar{p}\wedge\bar{q})+2\bar{k}\bar{q}}{\bar{q}(\alpha\vee(\beta+1))}}.
		\end{flalign*}
		Similarly, for $\bar{I}_{22}$, we have
		\begin{flalign*}
			\E \Big[\big(1+|X(t_k)|^{2\alpha}\big)|X(t_k)|^2I_{\{\gamma_{\phi}\leq T\}}\Big]\leq C\Delta^{\frac{\bar{k}(\bar{p}-2)(\bar{p}\wedge\bar{q})}{\bar{p}(\alpha\vee(\beta+1))}\wedge\frac{\bar{k}(\bar{p}\wedge\bar{q})(\bar{p}-2\alpha-2)}{\bar{p}(\alpha\vee(\beta+1))}}.
		\end{flalign*}
		Then it holds that
		\begin{flalign}\label{CD4}
			\E[|f(X(t_k))-f_\Delta(X(t_k))|^2]\leq C\Delta,
		\end{flalign}
		where we choose $\bar{k}=\frac{1}{2}$ and $\bar{p}=\bar{q}=2(\alpha\vee(\beta+1))+2\alpha+2$.  Therefore, we have
		\begin{flalign}\label{bI2}
			\bar{I}_2=\big\|f(X(t_k))-f_\Delta(X(t_k)) \big\|_{L^2(\Omega,\RR^d)}\Delta=\Delta\Big(\E[|f(X(t_k))-f_\Delta(X(t_k))|^2]\Big)^\frac{1}{2}\leq C\Delta^\frac{3}{2}. 
		\end{flalign}
		For $\bar{I}_3$, it follows from $\|\E[Y|\mathcal{F}_{t}]\|_{L^2(\Omega,\RR^d)}\leq\|Y\|_{L^2(\Omega,\RR^d)}$ for all $Y\in L^2(\Omega,\RR^d)$ that
		\begin{flalign*}
			\bar{I}_3=\int_{t_k}^{t_{k+1}}\big\| \E \big[f(X(s))-f(X(t_k)) | \mathcal{F}_{t_k} \big] \big\|_{L^2(\Omega,\RR^d)}\mathrm{d}s\leq\int_{t_k}^{t_{k+1}}\big\|f(X(s))-f(X(t_k)) \big\|_{L^2(\Omega,\RR^d)}\mathrm{d}s.
		\end{flalign*}
		Moreover, by Assumption \ref{as.1}, we obtain
		\begin{flalign*}
			&\E[|f(X(s))-f(X(t_k))|^2]\\
			\leq&K_1\E\Big[\Big(1+|X(s)|^\alpha+|X(t_k)|^\alpha+|X(s)|^{-\beta}+|X(t_k)|^{-\beta}\Big)^2|X(s)-X(t_k)|^2\Big]\\
			\leq&C\E\big[\big(1+|X(s)|^{2\alpha}+|X(t_k)|^{2\alpha}\big)|X(s)-X(t_k)|^2\big]+C\E\big[\big(1+|X(s)|^{-2\beta}+|X(t_k)|^{-2\beta}\big)|X(s)-X(t_k)|^2\big]\\
			:=&\bar{I}_{31}+\bar{I}_{32}.
		\end{flalign*}
		For $\bar{I}_{31}$, by Lemma \ref{Lm2.1} and the H$\rm \ddot{o}$lder inequlity, we obtain
		\begin{flalign*}
			&\E\big[\big(1+|X(s)|^{2\alpha}+|X(t_k)|^{2\alpha}\big)|X(s)-X(t_k)|^2\big]\\
			\leq&\Big(E\big[1+|X(s)|^{\bar{p}}+|X(t_k)|^{\bar{p}}\big]\Big)^\frac{2\alpha}{\bar{p}}\Big(\E\big[|X(s)-X(t_k)|^\frac{2\bar{p}}{\bar{p}-2\alpha} \big] \Big)^\frac{\bar{p}-2\alpha}{\bar{p}}\leq C\Big(\E\big[|X(s)-X(t_k)|^\frac{2\bar{p}}{\bar{p}-2\alpha} \big] \Big)^\frac{\bar{p}-2\alpha}{\bar{p}}.
		\end{flalign*}
		By letting $p=\frac{\bar{p}}{\bar{p}-2\alpha}$ in Lemma \ref{Lemma4.3}, we can obtain $\bar{I}_{31}\leq C\Delta$. Similarly, $\bar{I}_{32}\leq C\Delta $ can be also obtained. Therefore,
		\begin{flalign}\label{4.6}
			\E[|f(X(s))-f(X(t_k))|^2]\leq C\Delta.
		\end{flalign}
		Finally, we can get			
		\begin{flalign}\label{bI3}
			\bar{I}_3\leq\int_{t_k}^{t_{k+1}}\big\|f(X(s))-f(X(t_k)) \big\|_{L^2(\Omega,\RR^d)}\mathrm{d}s=\int_{t_k}^{t_{k+1}}\big(\E |f(X(s))-f(X(t_k))|^2 \big)^\frac{1}{2}\mathrm{d}s\leq C\Delta^\frac{3}{2}.
		\end{flalign}
		Substituting \eqref{bI1}, \eqref{bI2}, and \eqref{bI3} into \eqref{bI} yields that
		\begin{flalign}\label{Bcon1}
			\big\| \E \big[ X(t_k+\Delta)-\Psi(X(t_k),\Delta)| \mathcal{F}_{t_k} \big] \big\|_{L^2(\Omega,\RR^d)}\leq C\Delta^\frac{3}{2}.
		\end{flalign}
		It is not difficult to get
		\begin{flalign}\label{JJ}
			&\big\| \big( \id - \E [ \, \cdot \, | \mathcal{F}_{t} ] \big)\big( X(t + \Delta) - \Psi(X(t),\Delta) \big)\big\|_{L^2(\Omega,\RR^d)}\nonumber\\
			\leq &\int_{t_k}^{t_{k+1}}\big\|\big( \id - \E [ \, \cdot \, | \mathcal{F}_{t} ] \big)(f(X(s))-f(X(t_k))) \big\|_{L^2(\Omega,\RR^d)}\mathrm{d}s+\Big\|\big(g(X(t_k))-g_\Delta(X(t_k))\big)\Delta B_k\Big\|_{L^2(\Omega,\RR^{d\times m})}\nonumber\\
			&+\Big\|\int_{t_k}^{t_{k+1}}g(X(s))-g(X(t_k))\mathrm{d}B(s) \Big\|_{L^2(\Omega,\RR^{d\times m})}:=\bar{J}_1+\bar{J}_2+\bar{J}_3.
		\end{flalign}
		For $\bar{J}_1$, it follows from $\| \big( \id - \E [ \, \cdot \, | \mathcal{F}_{t} ] \big)Y\|_{L^2(\Omega,\RR^d)}\leq\|Y \|_{L^2(\Omega,\RR^d)}$ for all $Y\in L^2(\Omega,\RR^d) $ and the estimation of $\bar{I}_3$ that
		\begin{flalign}\label{J1}
			\bar{J}_1
			\leq &\int_{t_k}^{t_{k+1}}\big\|(f(X(s))-f(X(t_k))) \big\|_{L^2(\Omega,\RR^d)}\mathrm{d}s\leq C\Delta^\frac{3}{2}.
		\end{flalign}
		For $\bar{J}_2$, by the It\^o isometry and the similar proof of the estimation of $\eqref{CD4}$, we can get
		\begin{flalign}\label{J2}
			\bar{J}_2=\Delta^\frac{1}{2}\big\|g(X(t_k))-g_\Delta(X(t_k)) \big\|_{L^2(\Omega,\RR^{d\times m})}
			=\Delta^\frac{1}{2}\Big(\E[|g(X(t_k))-g_\Delta(X(t_k))|^2]\Big)^\frac{1}{2}\leq C\Delta. 
		\end{flalign}
			Applying the It\^o isometry and a similar way estimating to $\eqref{4.6}$ yields that
			\begin{flalign}\label{J3}
				\bar{J}_3 \leq\Big(\int_{t_k}^{t_{k+1}}\big\| g(X(t_k))-g_\Delta(X(t_k))\big\|_{L^2(\Omega,\RR^{d\times m})}^2\mathrm{d}s\Big)^\frac{1}{2}
				=\Big(\int_{t_k}^{t_{k+1}}\E |g(X(s))-g(X(t_k))|^2 \mathrm{d}s\Big)^\frac{1}{2}\leq C\Delta.
			\end{flalign}
			Substituting \eqref{J1}, \eqref{J2}, and \eqref{J3} into \eqref{JJ} yields that
			\begin{flalign}\label{Bcon2}
				\big\| \big( \id - \E [ \, \cdot \, | \mathcal{F}_{t} ] \big)\big( X(t + \Delta) - \Psi(X(t),\Delta) \big)\big\|_{L^2(\Omega,\RR^d)}\leq C\Delta.
			\end{flalign}
			Therefore, combining \eqref{Bcon1} with \eqref{Bcon2} yields that \eqref{bcon1} and \eqref{bcon2} hold with $\gamma_0=\frac{1}{2}$. \hfill
		\end{proof}
		
		\indent We conclude that the strong convergence rate of the TEM method is order 1/2, which follows directly from Theorem \ref{Thm4.2} and Theorem \ref{Thm4.3} as well as Theorem 3.7 in \cite{pj2016}. 
		
		\begin{proposition}\label{Pro 1}
			Let Assumptions \ref{as.1}, \ref{as.2}, and \ref{as.3} hold with $\alpha\vee(\beta+1)\leq \bar{p}+\bar{q}$. Then for any $\Delta\in(0,1)$, the TEM method is strongly convergent of $\frac{1}{2}$.
			\begin{equation}
				\sup_{0\leq k\Delta\leq T}\E [|X(t_k)-\tilde{X}_k|^2]\leq C\Delta.
			\end{equation}
		\end{proposition}
		
		\indent We present the following main theorem demonstrating that the PPTEM method is strongly convergent of order 1/2.
		\begin{theorem}\label{strmain}
			Let Assumptions \ref{as.1}, \ref{as.2}, and \ref{as.3} hold with $\alpha\vee(\beta+1)\leq \bar{p}+\bar{q}$. Then for all sufficiently small $\Delta\in(0,\Delta^*]$ satisfying \eqref{hD}, the PPTEM method is strongly convergent of order 1/2, i.e., 
			\begin{equation}\label{main}
				\sup_{0\leq k\Delta\leq T}\E[|X(t_k)-X_k|^2]\leq C\Delta.
			\end{equation}
		\end{theorem}
		\begin{proof}
			Recall $\tilde{\xi}_\phi:=\inf\{t_k\in[0,T]:\tilde{X}_{k,i}\notin((\phi^{-1}(h(\Delta)))^{-1},\phi^{-1}(h(\Delta)))\hspace{0.4em}\text{for some}\hspace{0.4em}i=1,2,\cdots,d\}$, where $\tilde{X}_{k,i}$ denotes the $i$-th element of $\tilde{X}_{k}$. Let $R=K_0\Delta^{-\frac{\bar{k}}{\alpha\vee(\beta+1)}}$, then $\phi^{-1}(h(\Delta))\geq R$. By the H\"older inequality, Lemma \ref{pp numerical integral} and Corollary \ref{Corollary2}, for all sufficiently small $\Delta\in(0,\Delta^*]$ satisfying \eqref{hD}, we have
			\begin{flalign*}
				&\E[|\tilde{X}_k-X_k|^2]=\E[|\tilde{X}_k-X_k|^2I_{\{\tilde{\xi}_\phi\leq T\}}] 
				\leq\E[|\tilde{X}_k|^2I_{\{\tilde{\xi}_\phi\leq T\}}]+\E[|X_k|^2I_{\{\tilde{\xi}_\phi\leq T\}}]\\
				\leq&\big(\E[|\tilde{X}_k|^{\bar{p}}]\big)^{\frac{2}{\bar{p}}}\big(\PP(\tilde{\xi}_\phi\leq T)^{\frac{\bar{p}-2}{\bar{p}}}\big)+\big(\E[|X_k|^{\bar{p}}]\big)^{\frac{2}{\bar{p}}}\big(\PP(\tilde{\xi}_\phi\leq T)^{\frac{\bar{p}-2}{\bar{p}}}\big)\\
				\leq&C\Delta^{\frac{\bar{k}(\bar{p}-2)(\bar{p}\wedge\bar{q})}{\bar{p}(\alpha\vee(\beta+1))}}.
			\end{flalign*}
			\indent Choosing $\bar{k}=\frac{1}{2}$ and $\bar{p}=\bar{q}=2(\alpha\vee(\beta +1))+2$, we get $\frac{\bar{k}(\bar{p}-2)(\bar{p}\wedge\bar{q})}{\bar{p}(\alpha\vee(\beta+1))}=1$. Consequently,
			\begin{flalign}\label{txx}
				\sup_{0\leq k\Delta\leq T}\E[|\tilde{X}_k-X_k|^2]\leq C\Delta.
			\end{flalign}
			Finally, by Proposition \ref{Pro 1} and \eqref{txx}, we derive
			\begin{flalign*}
				\sup_{0\leq k\Delta\leq T}\E[|X(t_k)-X_k|^2]\leq\sup_{0\leq k\Delta\leq T}\E[|X(t_k)-\tilde{X}_k|^2+|\tilde{X}_k-X_k|^2]\leq C\Delta.
			\end{flalign*}
			The proof is completed.\hfill
		\end{proof}

		\section{Numerical experiments}
		\label{6}
		In this section, we present numerical experiments conducted by using the PPTEM method, along with numerical comparisons between the PPTEM scheme and the TEM \cite{M15} and EM methods to validate our theoretical results. Before conducting numerical experiments, we need to provide some instructions. The expressions for the evaluated mean square error at the terminal time $T$ are given by 
		\begin{equation*}
			\Vert X(T)-X_T \Vert_{L_{2}} = \Big( \frac{1}{M}\sum_{i=1}^{M}|X^i(T) - X^i_T|^2\Big)^{\frac{1}{2}},
		\end{equation*}
		where $M=10^5$ denote the number of sample paths, $X^i(T)$ denotes the $i$-th exact solution, $X^i_T$ denotes the $i$-th numerical solution. Unless otherwise specified, when we estimate the mean square errors of a method, we usually use the numerical solutions from this method with step size $\Delta = 2^{-14}$ as a replacement for the unknown exact solution, and generate the numerical solutions of this method with different step sizes $\Delta = 2^{-12}, 2^{-11}, 2^{-10}, 2^{-9}, 2^{-8}$. 
		
		\subsection{One-dimensional tests}

		\begin{example}{CEV process}\label{CEV}
			\par We consider the CEV process
			\begin{align}\label{cev}
				\mathrm{d}X(t) = \kappa(\mu-X(t))\mathrm{d}t + \xi X^\theta(t) \mathrm{d}B(t),
			\end{align}
			where constants $\kappa,\mu,\xi>0$ and $\theta \in (\frac{1}{2},1)$. According to Example 4.2 in \cite{TX2024}, by using the Lamperti transformation $Y=X^{1-\theta}$, the coefficients of transformed SDE satisfy Assumptions \ref{as.1}, \ref{as.2}, and \ref{as.3}. This means that Theorem \ref{strmain} can be applied to the CEV process. Therefore, for all $\Delta\in(0,\Delta^*]$, we have 
			\begin{flalign*}
				\sup_{0\leq k\Delta\leq T}\E[|Y(t_k)-Y_k|^2]\leq C\Delta,
			\end{flalign*}
			where $Y_k$ is the numerical solution of PPTEM method. Moreover, one can show that the PPTEM method has a strong convergence order 1/2 for this model. \\
			\indent In our experiments, we take $\kappa = 4, \mu = 0.5, \xi = 1, \alpha = 0.55, T = 1, X_0 = 2$. As shown in Fig. \ref{figure1}, the strong convergence rate is $\frac{1}{2}$, which confirms the theoretical results. 
		\end{example}
		
		\begin{figure}[htbp]
			\centering
			
			\includegraphics[height = 8cm, width=10cm]{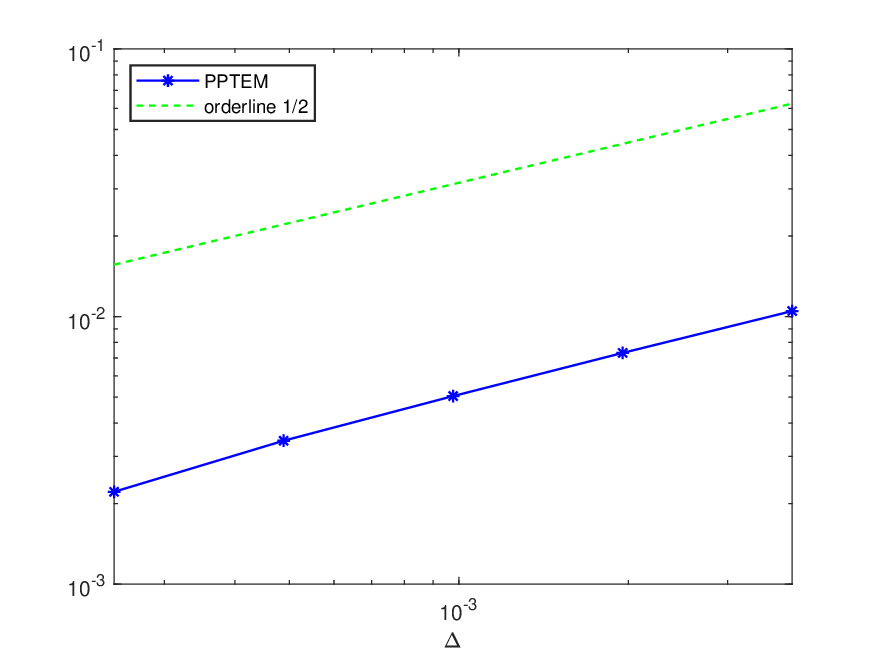}		
			\caption{Mean square convergence rate of the PPTEM method for the CEV model.}
			\label{figure1}
		\end{figure}
		
		\begin{example}{A\"it--Sahalia (AS) model}\label{AST}
			\par We consider the generalized AS model
			\begin{align}\label{ast}
				\mathrm{d}X(t) = [a_{-1}X^{-1}(t)-a_0+a_1X(t)-a_2X^r(t)]\mathrm{d}t + \sigma X^\rho(t) \mathrm{d}B(t),
			\end{align}
			where constants $a_{-1},a_0,a_1,a_2,\sigma>0$, $r,\rho> 1$ and $r+1>2\rho$. According to Example 4.1 in \cite{TX2024}, the coefficients of SDE \eqref{ast} satisfy Assumptions \ref{as.1}, \ref{as.2}, and \ref{as.3}. Therefore, Theorem \ref{strmain} can be applied to the AS model, which implies that the PPTEM method has a strong convergence order 1/2 for this model. \\
			\indent In a special case of the generalized Ait-Sahalia process with parameters $(a_{-1}, a_0, a_1, a_2, r, \rho) = (0, 0, \lambda + \frac{\sigma}{2}, 1, 3, 1)$, the scalar stochastic Ginzburg-Landau (GL) equation 
			\begin{flalign*}
				\mathrm{d}X(t) = (-X^3(t)+(\lambda+\frac{\sigma^2}{2})X(t))\mathrm{d}t+\sigma X(t)\mathrm{d}B(t),
			\end{flalign*}
			where $\lambda,\sigma\geq0$, complies with the aforementioned assumptions. \\
			\indent In our experiments, we take $\lambda = 1, \sigma = 5, T = 1, X_0 =1$. As shown in Fig. \ref{figure2}, we observe that the strong convergence rate of the PPTEM method is $\frac{1}{2}$, which is consistent with our theoretical results. \\
			\indent Besides, we take $a_{-1}=3,a_0=2,a_1=1,a_2=5,\sigma=2,r=4,\rho=2,X_0=2,T=1$ such that $r>2\rho-1$. In Fig. \ref{figure3}, the strong convergence rate is $\frac{1}{2}$ as expected. 
		\end{example}
		
		\begin{figure}[htbp]
			\centering
			\includegraphics[height = 8cm, width=10cm]{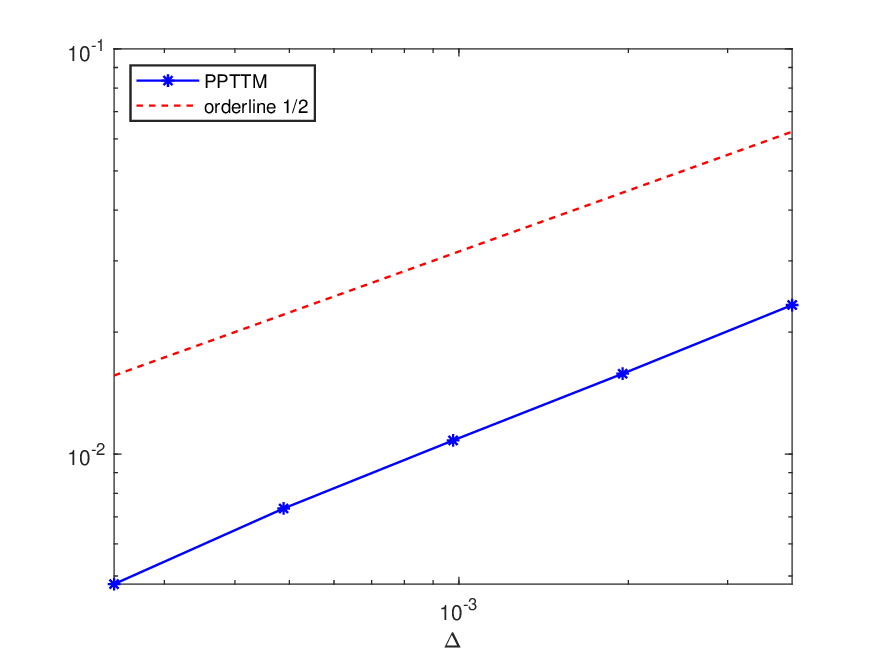}		
			\caption{Mean square convergence rates of the PPTEM method for the GL model.}
			\label{figure2}
		\end{figure}
		
		\begin{figure}[htbp]
			\centering
			\includegraphics[height = 8cm, width=10cm]{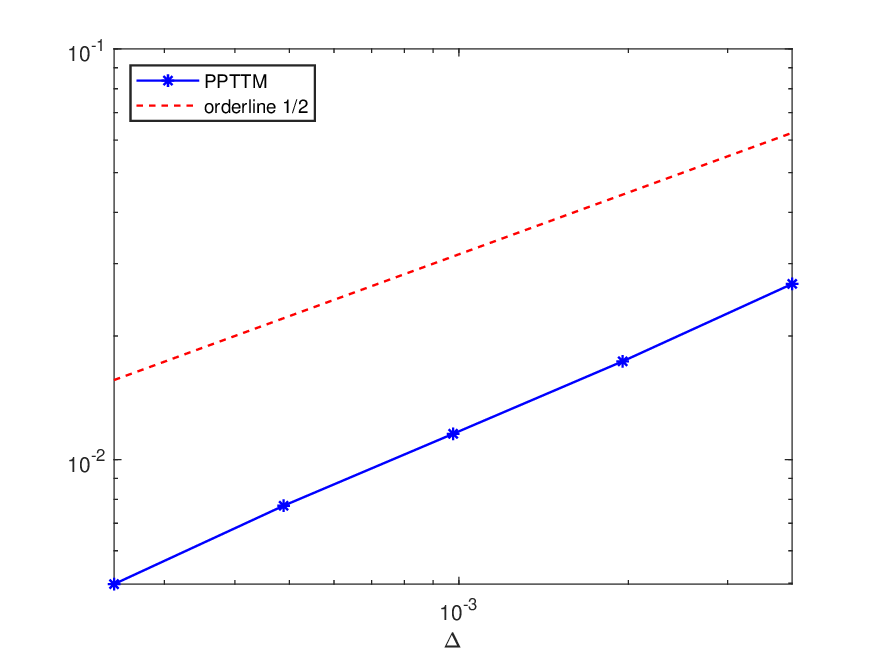}
			\caption{Mean square convergence rates of the PPTEM method for the A\"it--Sahalia model.}
			\label{figure3}
		\end{figure}
		
		\indent We list the percentages of non-positive numerical values of the PPTEM, EM and TEM methods for the two models, see Table \ref{tabb}. As anticipated, our numerical method preserves the positivity of the numerical solutions, whereas the other methods do not. Positivity preserving is the advantage of the PPTEM method, which is consistent with the theoretical result.
		
		\begin{table}[htbp]
			\caption{The percentages of non-positive numerical values of the PPTEM, EM, and TEM methods with different step sizes using $10^5$ sample paths for two models\label{tabb}}  
			\begin{tabular*}{\columnwidth}{@{\extracolsep\fill}lllll@{\extracolsep\fill}@{\extracolsep\fill}}
				\toprule
				Model &Step sizes & PPTEM (\%) &EM (\%)  &TEM (\%) \\
				\midrule
				CEV & \tabincell{c}{ $2^{-2}$ \\ $2^{-3}$ \\ $2^{-4}$ \\ $2^{-5}$ } & \tabincell{c}{0\\0\\0\\0} & \tabincell{c}{12.09\\1.47\\0.12\\0.01} & \tabincell{c}{6.19\\0.11\\0\\0}\\
				\hline
				AS & \tabincell{c}{ $2^{-6}$ \\ $2^{-7}$ \\ $2^{-8}$\\ $2^{-9}$ } & \tabincell{c}{0\\0\\0\\0} & \tabincell{c}{8.93\\1.88\\0.25\\0.02} & \tabincell{c}{27.85\\0.50\\0.01\\0} \\
				\bottomrule  
			\end{tabular*}
		\end{table}
		
		\par Compared with the logarithmic TEM method \cite{PPlogTM,LG2023} and the logarithmic truncated Milstein method \cite{HXW}, authors in \cite{TX2024} employ weaker conditions such that the parameter restrictions have been lifted. To be more specific, the A\"it--Sahalia model with $r>2\rho-1$ rather than $r>4\rho-3$ and the CEV process with $\theta\in(\frac{1}{2},1)$ rather than $\theta\in(\frac{3}{4},1)$. It is worth mentioning that our method also achieves this, except in terms of the assumptions regarding strong convergence. Besides, our method still maintains its advantages in the following multi-dimensional numerical examples.

		\subsection{Multi-dimensional tests}
		
		\begin{example}{Stochastic Lotka--Volterra system }\label{6.1}
			\par We consider the $d$-dimensional stochastic Lotka--Volterra system 
			\begin{flalign}\label{LV}
				\mathrm{d}X(t)&=\mathrm{diag}(X_1(t),\cdots,X_d(t))[f(X(t))\mathrm{d}t+(\sigma+\zeta(X(t)))\mathrm{d}B(t)] \nonumber\\
				&:=F(X(t))\mathrm{d}t+G(X(t))\mathrm{d}B(t),
			\end{flalign}
			where $f(X)=(f_1(X),\cdots,f_d(X))^T=c+AX:\RR_+^d\rightarrow\RR^d$, the parameters $c=(c_1,\cdots,c_d)^T\in\RR^d$, $A=(a_{ij})_{d\times d}\in\RR^{d\times d}$, $\sigma=(\sigma_1,\cdots,\sigma_d)^T\in\RR^d$ and $\zeta=(\zeta_1,\cdots,\zeta_d)^T:\RR_+^d\rightarrow\RR^d$. \\
			\indent For any $a,b\in\RR_+^d$, we let $\mathrm{L}(a,b):=\{a+t(b-a)|t\in[0,1]\}$. The mean value theorem implies that there exists $u\in \mathrm{L}(a,b)$ such that 
			\begin{flalign*}
				F(a)-F(b)=DF(u)(a-b).
			\end{flalign*}
			Then it follows that 
			\begin{flalign*}
				|F(a)-F(b)|\leq|DF(u)||a-b|\leq C(1+|a|+|b|)|a-b|
			\end{flalign*}
			because $DF(X)=c+2\mathrm{diag}(X_1,\cdots,X_d)A$. It is easy to see that $|DG(X)|\leq C$ under Assumption 2.3 in \cite{YXF2024}. Now we see Assumption \ref{as.1} holds with $\alpha=1$ and $\beta=0$.\\
			\indent Under Assumption 4.1 in \cite{YXF2024}, one can verify that Assumption \ref{as.2} can be satisfied with $c_i\geq -a_{ii}+\frac{\bar{q}+1}{2}(\sigma_i+\zeta_i)^2)$ for any $i\in\{1,2,\cdots,d\}$.
			Further, there exists a positive constant $C$ such that
			\begin{flalign}\label{LVass2}
				X^TF(X)+\frac{\bar{p}-1}{2}|G(X)|^2\leq C(1+|X|^2), \quad |X|\in[\bar{x},\infty),
			\end{flalign}
			because \eqref{LVass2} tends to negative infinite as $|X|\rightarrow \infty$. Moreover, for any $p>2$,
			\begin{flalign*}
				\langle DF(X)y,y\rangle+\frac{p-1}{2}|DG(X)y|^2 &=\langle  [c+2\mathrm{diag}(X_1,\cdots,X_d)A]y,y\rangle+\frac{p-1}{2}|[\sigma+\zeta(X)]y|^2\\
				&\leq C|y|^2.
			\end{flalign*}
			Thus, Assumption \ref{as.3} holds. It means that Theorem \ref{strmain} can be applied to the multi-dimensional Lotka--Volterra system. That is to say, for this model, the PPTEM method has the strong convergence order 1/2.\\
			\indent For our experiments, we take $c_1=50,c_2=30,c_3=20,a_{11}=-55,a_{22}=-10,a_{33}=-15,\sigma_1=7,\sigma_2=2,\sigma_3=5, T=1, X_0=(0.5,2,1)^T, \zeta_1=\frac{\sin(X_1(t))+\sin(X_2(t))+\sin(X_3(t))}{1+X_1(t)+X_2(t)+X_3(t)},\zeta_2=\frac{X_1(t)+X_2(t)+X_3(t)}{1+(X_1(t)+X_2(t)+X_3(t))^2}$ and $\zeta_3= \frac{\cos(X_1(t))+\cos(X_2(t))}{1+X_3^2(t)}$. Besides, we take other non-mentioned parameters as zero. Then we have the following equations (e.g. in \cite{YXF2024})
			\begin{flalign}\label{LV3}
				&\mathrm{d}X_1(t)=\big(50X_1(t)-55X^2_1(t)\big)\mathrm{d}t+X_1(t)\Big(7+\frac{\sin(X_1(t))+\sin(X_2(t))+\sin(X_3(t))}{1+X_1(t)+X_2(t)+X_3(t)}\Big)\mathrm{d}B(t),\nonumber\\
				&\mathrm{d}X_2(t)=\big(30X_2(t)-10X^2_2(t)\big)\mathrm{d}t+X_2(t)\Big(2+\frac{X_1(t)+X_2(t)+X_3(t)}{1+(X_1(t)+X_2(t)+X_3(t))^2}\Big)\mathrm{d}B(t),\\
				&\mathrm{d}X_3(t)=\big(20X_3(t)-15X^2_3(t)\big)\mathrm{d}t+X_3(t)\Big(5+\frac{\cos(X_1(t))+\cos(X_2(t))}{1+X_3^2(t)}\Big)\mathrm{d}B(t)\nonumber
			\end{flalign}
			in $\RR_+^3$. As shown in Fig. \ref{figure4}, it is clearly seen that the strong convergence rate is consistent with our theoretical results.
		\end{example}
		
		\begin{figure}[htbp]
			\centering
			
			\includegraphics[height = 8cm, width=10cm]{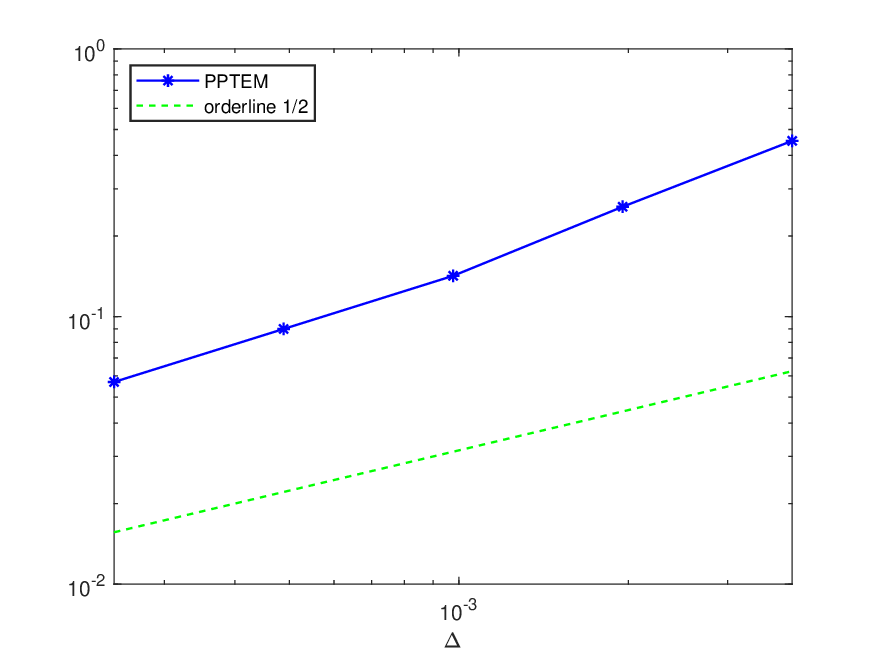}		
			
			\caption{Mean square convergence rate of the PPTEM method for the stochastic Lotka-Volterra system. }
			\label{figure4}
		\end{figure}
		
		Besides, we list the mean-square error of the EM, TEM, and PPTEM methods about Example \ref{6.1}. Table \ref{tab} shows that the EM and TEM methods fail to solve the SDE \eqref{LV} effectively when using a large step size. In contrast, our method effectively solves this equation.
		
		\begin{table}[htbp]
			\caption{The mean-square error of the EM, TEM and PPTEM methods with different step sizes using $10^5$ sample paths\label{tab}}  
			\begin{tabular*}{\columnwidth}{@{\extracolsep\fill}llll@{\extracolsep\fill}llll@{\extracolsep\fill}llll@{\extracolsep\fill}}
				\toprule 
				Step sizes &$2^{-8}$ &$2^{-9}$  &$2^{-10}$ &$2^{-11}$ &$2^{-12}$  \\
				\midrule
				EM& NaN & NaN &NaN & 0.0914&  0.0572 \\
				TEM&  NaN &NaN&NaN & 0.0893&  0.0565\\
				PPTEM & 0.4538 &0.2570 & 0.1417& 0.0897&0.0568\\
				\bottomrule
			\end{tabular*}
		\end{table}
		
		The positivity-preserving methods \cite{YQX2024,YXF2024} have been proposed to solve multi-dimensional stochastic Kolmogorov equations with super-linear coefficients. In the case where $d=1$, these stochastic Kolmogorov equations can be simplified to
		\begin{flalign*}
			\mathrm{d}x(t)=x(t)[(c+ax(t)) \mathrm{d}t+ (\sigma+\zeta(x(t)))\mathrm{d}B(t)],
		\end{flalign*}
		which excludes the SDE presented in equation \eqref{ast} when $r>2$, our method is capable of solving \eqref{ast}. As a result, we see that our method can be applied to a broader range of model equations compared to the existing multi-dimensional positivity-preserving methods. 
		
		\begin{example}{Susceptible-Infected-Recovered-Susceptible (SIRS) epidemic model}\label{SIR1}
			\par We consider SIRS epidemic model
			\begin{flalign}\label{SIR}
				\left\{ 
				\begin{aligned}
					&\mathrm{d}S(t) = [\mu(t)-\mu(t) S(t)-\xi(t) S(t)I(t)]\mathrm{d}t-\sigma S(t)I(t)\mathrm{d}B(t), \\
					&\mathrm{d}I(t) = [\xi(t) S(t)I(t)-(\mu(t)+\gamma(t)) I(t)]\mathrm{d}t+\sigma S(t)I(t)\mathrm{d}B(t),\\
					&\mathrm{d}R(t) =[\gamma(t) I(t)-\mu(t) R(t)]\mathrm{d}t,
				\end{aligned}
				\right.
			\end{flalign}
			where $\sigma>0$ and $\mu(t), \xi(t), \gamma(t)$ are positive, non-constant and continuous functions of period $\omega$ with $\gamma(t)<\mu(t) $ for all $t\geq 0$. The model \eqref{SIR} is the version without time delay in \cite{SCZ2020}. Meanwhile, as mentioned in \cite{SCZ2020}, it holds $N(t)=S(t)+I(t)+R(t)\leq A$, where $N(t)$ denotes the population sizes and $A:=N(0)\vee1$. To better verify the hypothesis, we rewrite equation \eqref{SIR} as the following equation
			\begin{flalign*}
				\mathrm{d}X(t)=F(X(t))\mathrm{d}t+G(X(t))\mathrm{d}B(t).
			\end{flalign*}
			\indent For any $a,b\in\RR_+^3$, we let $\mathrm{L}(a,b):=\{a+t(b-a)|t\in[0,1]\}$. The mean value theorem implies that there exists $u\in \mathrm{L}(a,b)$ such that 
			\begin{flalign*}
				F(a)-F(b)=DF(u)(a-b).
			\end{flalign*}
			It follows that 
			\begin{flalign*}
				|F(a)-F(b)|^2&= |DF(u)|^2|a-b|^2\\
				&\leq [(\mu+\xi u_2)^2+(\xi u_1-(\mu+\gamma))^2+\mu^2]|a-b|^2\\
				&\leq [3\mu^2+2\xi^2u_2^2+2\xi^2u_1^2+2(\mu+\gamma)^2]|a-b|^2\\
				&\leq C(1+|a|^2+|b|^2)|a-b|^2.
			\end{flalign*}
			Similarly, one can see 
			\begin{flalign*}
				|G(a)-G(b)|^2\leq C(1+|a|^2+|b|^2)|a-b|^2.
			\end{flalign*}
			Then we see Assumption \ref{as.1} holds with $\alpha=\beta=1$. \\
			\indent For Assumption \ref{as.2}, we need examine the conditions one by one. We can always find a sufficiently small $\bar{x}_1>0$ such that
			\begin{flalign*}
				\mu S - \mu S^2- \xi S^2I-\frac{\bar{q}+1}{2}\sigma^2S^2I^2 \geq 0
			\end{flalign*}
			for any $S\in(0,\bar{x}_1)$. And we have
			\begin{flalign*}
				\mu S - \mu S^2- \xi S^2I+\frac{\bar{p}-1}{2}\sigma^2S^2I^2 \leq C S +CS^2 \leq C(1+S^2),
			\end{flalign*}
			where $I\leq N \leq N(0)\vee1$ was used. Next, in order for
			\begin{flalign*}
				&\xi SI^2-(\mu+\gamma)I^2-\frac{\bar{q}+1}{2}\sigma^2S^2I^2\\
				=&I^2(\xi S-\mu-\gamma-\frac{\bar{q}+1}{2}\sigma^2S^2)\\
				=&-I^2(\frac{\bar{q}+1}{2}\sigma^2S^2-\xi S+\mu+\gamma)\geq 0
			\end{flalign*}
			to hold, we need 
			\begin{flalign*}
				\frac{\bar{q}+1}{2}\sigma^2S^2-\xi S+\mu+\gamma \leq 0,
			\end{flalign*}
			which implies
			\begin{flalign*}
				\xi^2-2(\bar{q}+1)(\mu+\gamma)\geq0.
			\end{flalign*}
			Similar to the above validation, we can finally conclude that Assumption \ref{as.2} holds with $\xi^2(t)\geq2(\bar{q}+1)(\mu(t)+\gamma(t))$. \\
			\indent For any $p>2$, 
			\begin{flalign*}
				&\langle DF(X)y,y\rangle+\frac{p-1}{2}|DG(X)y|^2 \\
				=&(-\mu-\xi I)y_1^2+(\xi S-(\mu+\gamma))y_2^2-\mu y_3^2+\sigma^2I^2y_1^2+\sigma^2S^2y_2^2\\
				\leq& \xi Sy_2^2+\sigma^2I^2y_1^2+\sigma^2S^2y_2^2\leq C|y|^2,
			\end{flalign*}
			which means that Assumption \ref{as.3} holds. The verification above means that Theorem \ref{strmain} can be applied to model \eqref{SIR}.\\
			\indent For our experiments, we take $\mu(t)=1+0.3\cos(\pi t), \xi(t)=2+\sin(\pi t), \gamma(t)=0.6+0.1\sin(\pi t), \sigma = 1, S(0)= 3, I(0)=0.5, R(0)=0.5$. As shown in Fig. \ref{figure5}, it is seen that the strong convergence rate is consistent with our theoretical results. On the other hand, we plot individual trajectories of the TEM method and the PPTEM method in Fig. \ref{figure6}. We notice that our method preserves the positivity of the numerical solutions, while the TEM method fails to achieve this.
		\end{example}
		
		\begin{figure}[htbp]
			\centering
			\includegraphics[height = 8cm, width=10cm]{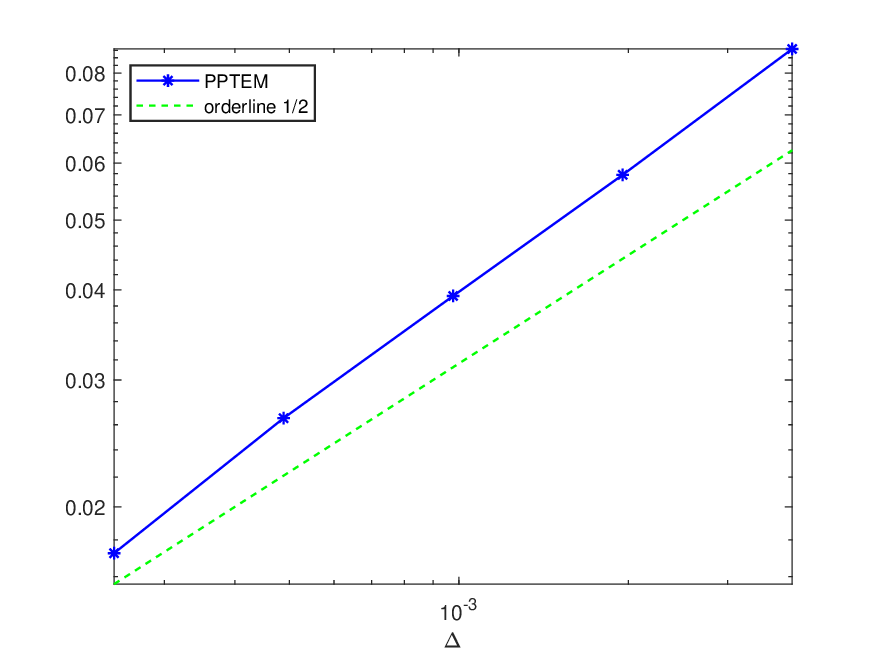}
			
			\caption{Mean square convergence rate of the PPTEM method for the SIRS epidemic model.}
			\label{figure5}
		\end{figure}
		\begin{figure}[htbp]
			\centering
			\subfigure{\label{SIR_path1}
				\includegraphics[height = 3cm, width=4cm]{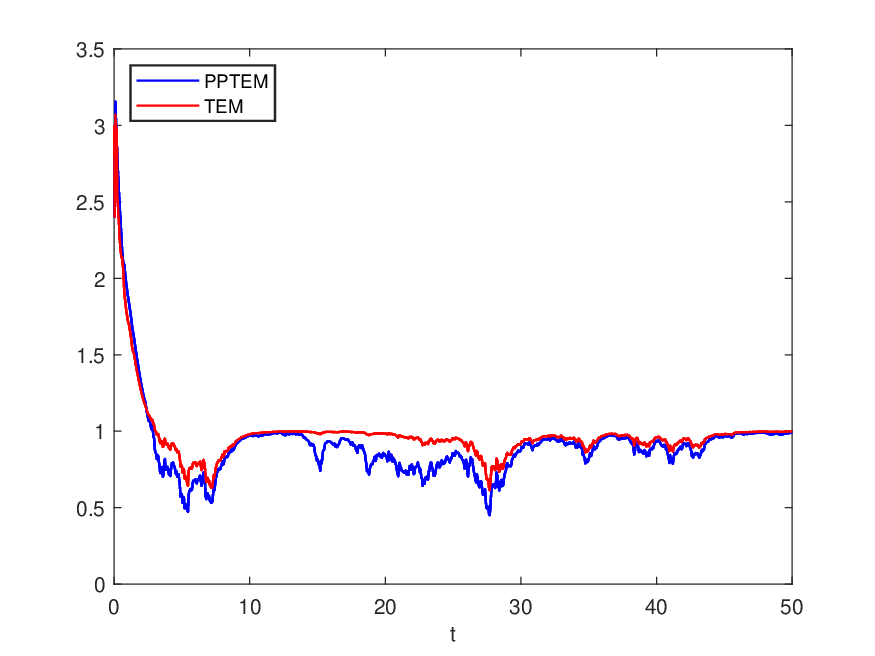}}		
			\subfigure{\label{SIR_path2}
				\includegraphics[height = 3cm, width=4cm]{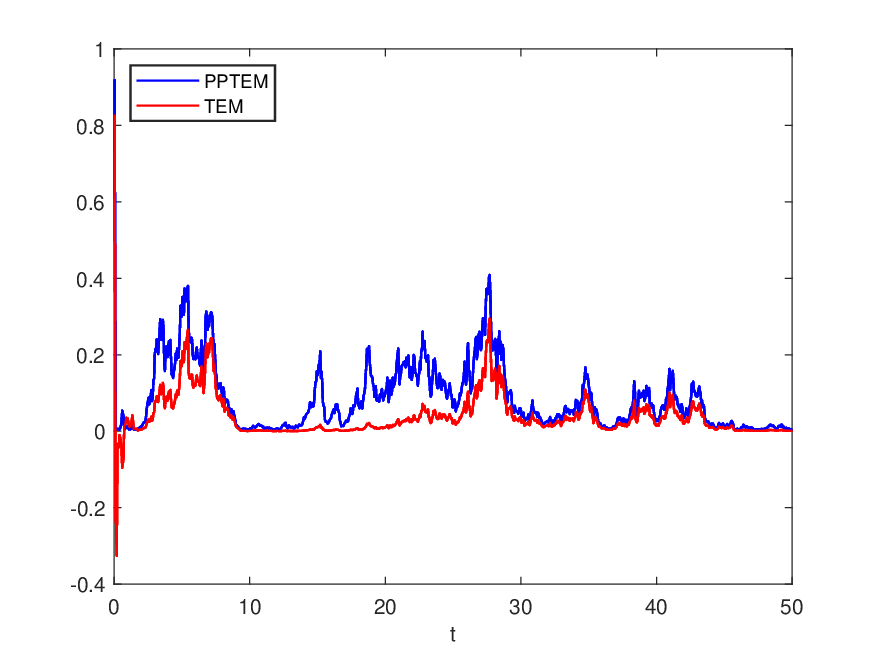}}
			\subfigure{\label{SIR_path3}
				\includegraphics[height = 3cm, width=4cm]{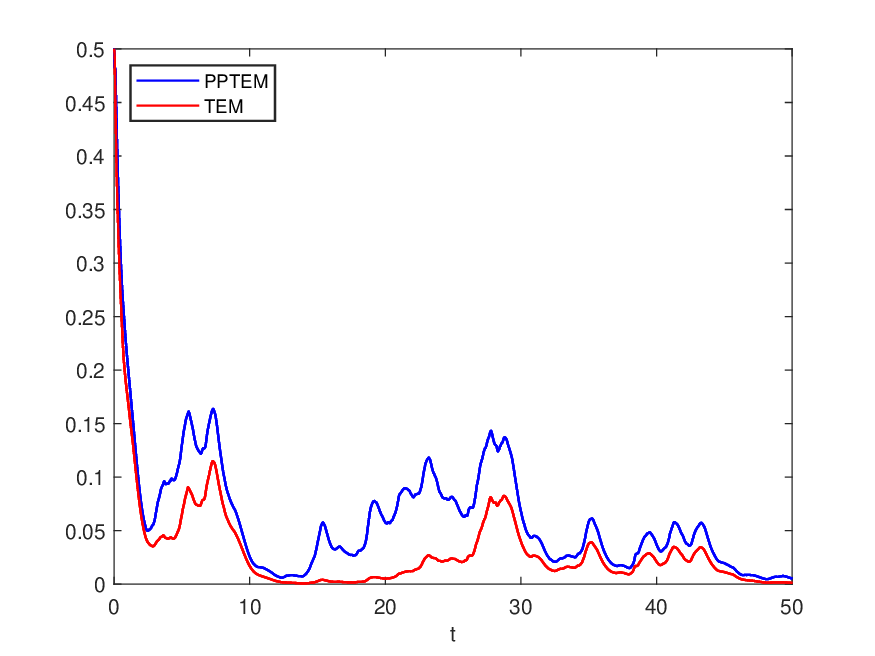}}
			\caption{Sample paths of SIRS epidemic model in Example \ref{SIR1}. $\Delta=0.05$; $T=50$. Left: S(t); Middle: I(t); Right: R(t).}
			\label{figure6}
		\end{figure}
		
		\begin{example}{Stochastic HIV/AIDS model}\label{HIV1}
			\par We consider stochastic HIV/AIDS model \cite{RHZ2025}
			\begin{flalign}\label{HIV}
				\left\{ 
				\begin{aligned}
					&\mathrm{d}S(t) = [N-\mu_1(t) S(t)-\xi S(t)I(t)]\mathrm{d}t-\sigma S(t)I(t)\mathrm{d}B(t), &t\in(0,T],\\
					&\mathrm{d}I(t) = [\xi S(t)I(t)-\mu_1(t)I(t)-\gamma(t) I(t)]\mathrm{d}t+\sigma S(t)I(t)\mathrm{d}B(t),&t\in(0,T],\\
					&\mathrm{d}A(t) =[\gamma(t) I(t)-\mu_1(t) A(t)-\mu_2(t)A(t)]\mathrm{d}t,&t\in(0,T],\\
					&S(0)=S_0,\quad I(0)=I_0,\quad A(0)=A_0,&t=0,
				\end{aligned}
				\right.
			\end{flalign}
			where $\mu_1,\mu_2,\xi,\gamma,\sigma>0$ and $S_0,I_0,A_0\geq0$. According to Theorem 2.2 in \cite{RHZ2025}, it holds that $U(t)=S(t)+I(t)+A(t)\leq G$, where $G$ is a positive constant. This along with similar verification in Example \ref{SIR1} shows that model \eqref{HIV} meet Assumptions \ref{as.1}, \ref{as.2} and \ref{as.3}.\\
			\indent For our experiments, we take $N=1, \mu_1(t)=0.5, \mu_2(t)=0.4, \xi=0.5, \gamma(t)=0.3, \sigma = 1, S(0)= 2, I(0)=1, R(0)=1$. As illustrated in Fig. \ref{figure7}, the observed strong convergence rate align with our theoretical predictions. Additionally, we present single trajectories for both the TEM method and the PPTEM method in Fig. \ref{figure8}. The results demonstrate that our method offers advantages over the TEM method in maintaining positivity.
		\end{example}

		\begin{figure}[htbp]
			\centering
			
			\includegraphics[height = 8cm, width=10cm]{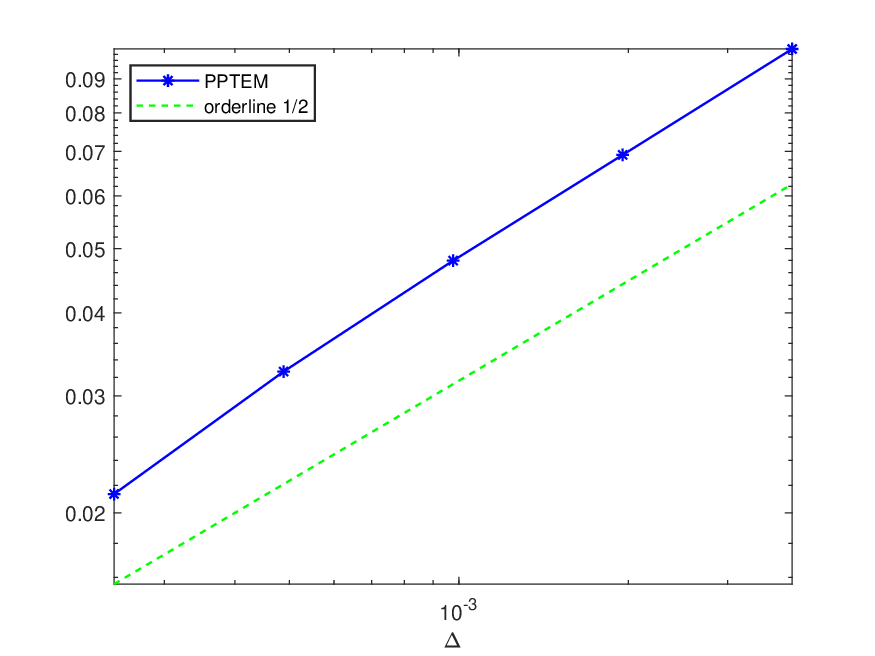}		
			\caption{Mean square convergence rate of the PPTEM method for the stochastic HIV/AIDS model. }
			\label{figure7}
		\end{figure}
		\begin{figure}[htbp]
			\centering
			\subfigure{\label{HIV_path1}
				\includegraphics[height = 3cm, width=4cm]{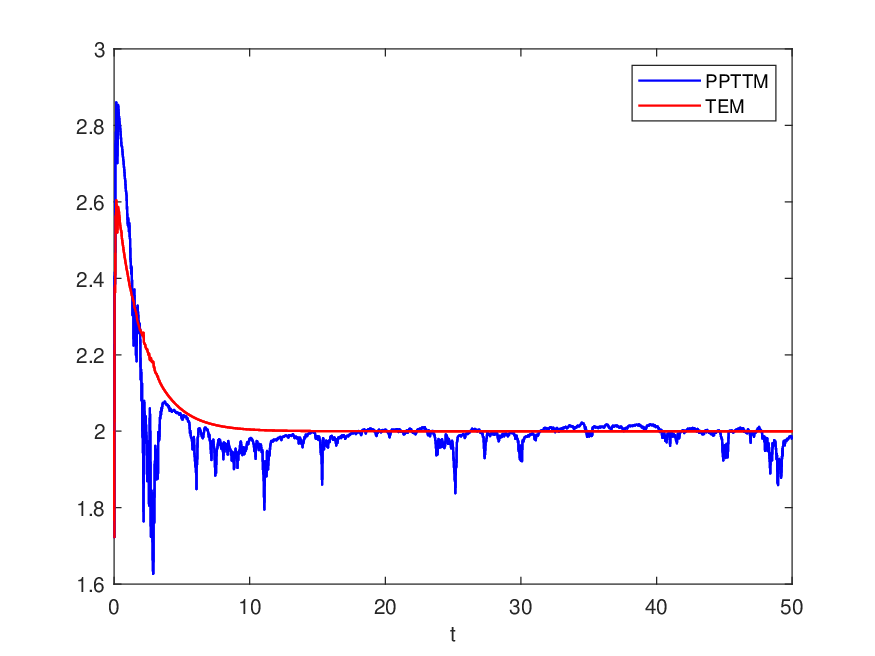}}		
			\subfigure{\label{HIV_path2}
				\includegraphics[height = 3cm, width=4cm]{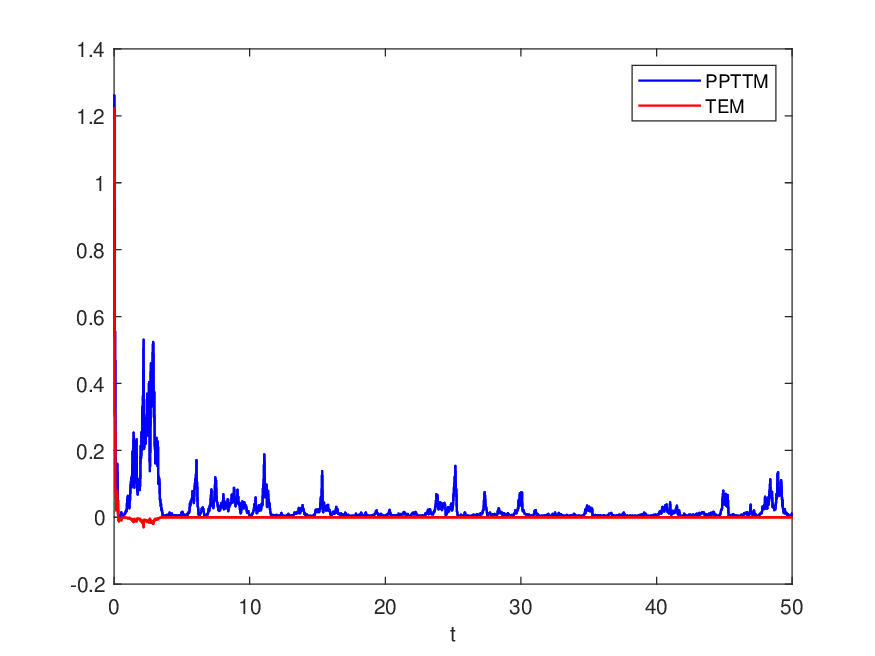}}
			\subfigure{\label{HIV_path3}
				\includegraphics[height = 3cm, width=4cm]{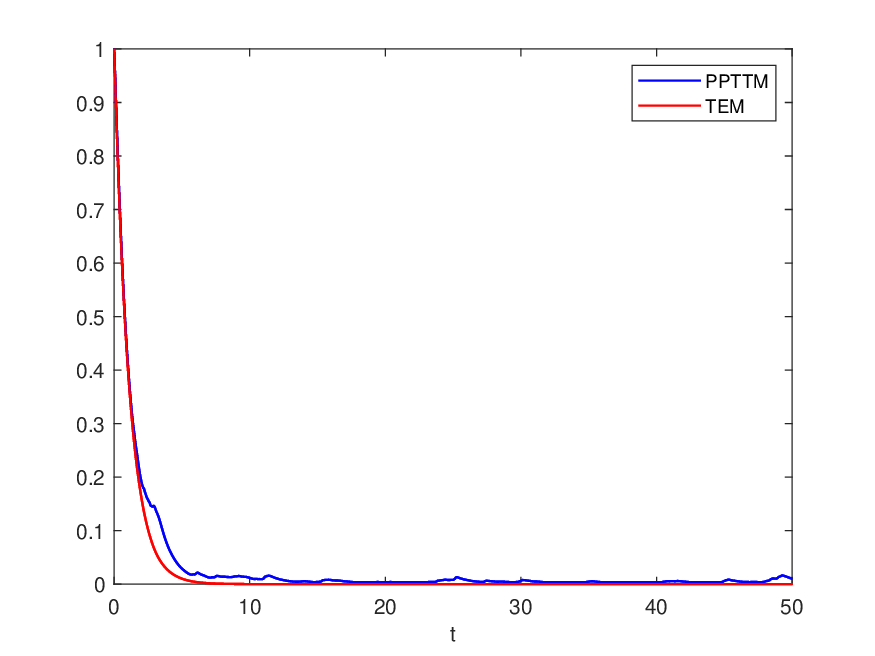}}
			\caption{Sample paths of stochastic HIV/AIDS model in Example \ref{HIV1}. $\Delta=0.02$; $T=50$. Left: S(t); Middle: I(t); Right: A(t).}
			\label{figure8}
		\end{figure}
		
		\par In conclusion, it is important to highlight that the theoretical strong convergence rates of Example \ref{6.1} in \cite{YXF2024} and Example \ref{HIV1} in \cite{RHZ2025} are approximately 1/2 and 1/4, respectively. Our method achieves a strong convergence rate of 1/2, both theoretically and through numerical tests.

		\section{Conclusion}
		
		\indent In this paper, we present the PPTEM method  designed for general multi-dimensional SDEs with positive solutions. We provide a detailed analysis of both the strong error associated with this method. Specifically, we demonstrate that the optimal strong convergence rate is of order 1/2. Our numerical examples confirm that the PPTEM method retains its theoretical accuracy and effectively preserves the positivity of the numerical solutions. Future research will focus on developing a higher-accuracy positive-preserving numerical method for solving multi-dimensional SDEs with positive solutions.\\

		\appendix
		\renewcommand{\thesection}{Appendix \Alph{section}}
		\renewcommand{\theequation}{\Alph{section}.\arabic{equation}}
		\section{Proof of Lemma 4.1}
		\label{Appendix11}
		Using the H\"older inequality yields that
		\begin{flalign*}
			&\E[\sup_{ t \in[0,T]}|\tilde{Y}(t)-Y(t)|^p]\\
			\leq&\E[\max_{0\leq k\leq N}\sup_{t_k\leq t\leq t_{k+1}}|f_r(Y_k)(t-t_k)+g_r(Y_k)(B(t)-B(t_k))|^p]\\
			\leq&2^{p-1}\E\Big[\max_{0\leq k\leq N}\sup_{t_k\leq t\leq t_{k+1}}|f_r(Y_k)|^p(t-t_k)^p+|g_r(Y_k)|^p|B(t)-B(t_k)|^p\Big]\\
			\leq&C_r \Delta^p+C_r\Big(\E[\max_{0\leq k\leq N}\sup_{t_k\leq t\leq t_{k+1}}|B(t)-B(t_k)|^{2n}]\Big)^\frac{p}{2n}.
		\end{flalign*}
		It follows from the Doob martingale inequality that
		\begin{flalign*}
			&\E[\max_{0\leq k\leq N}\sup_{t_k\leq t\leq t_{k+1}}|B(t)-B(t_k)|^{2n}]\\
			\leq&\sum_{k=0}^{N}\E[\sup_{t_k\leq t\leq t_{k+1}}|B(t)-B(t_k)|^{2n}]\leq\Big(\frac{2n}{2n-1}\Big)^{2n}\sum_{k=0}^{N}\E[|B(t_{k+1})-B(t_k)|^{2n}]\\
			\leq&\Big(\frac{2n}{2n-1}\Big)^{2n}\sum_{k=0}^{N}(2n-1)!!\Delta^n\leq\Big(\frac{2n}{2n-1}\Big)^{2n}(2n-1)!!(T+1)\Delta^{n-1},
		\end{flalign*}
		where $(2n-1)!!=(2n-1)\times(2n-3)\times\cdots\times3\times1$. By using \eqref{con3.3} and the inequality
		\begin{flalign*}
			((2n-1)!!)^\frac{1}{n}\leq \frac{1}{n}\sum_{i=0}^{n}(2i-1)=n,
		\end{flalign*}
		we have
		\begin{flalign*}
			&\E[\sup_{ t \in[0,T]}|\tilde{Y}(t)-Y(t)|^p]
			\leq C_r\Delta^p+C_rn^{\frac{p}{2}}\Delta^{\frac{p(n-1)}{2n}}.
		\end{flalign*}
		The assertion \eqref{as3.3} holds.
		
		\section{Proof of Lemma 4.6}\label{Appendix12}
		Denote $\widetilde{\Delta} = \Delta^{-\frac{1}{2\alpha}}\vee \Delta^{-\frac{1}{2(\beta+1)}}$.
		Using Assumption \ref{as.3} yields that 
		\begin{flalign*}
			&|\pi_\Delta(x)-\pi_\Delta(y)+(f_\Delta(x)-f_\Delta(y))\Delta|^2+\mu\Delta|g_\Delta(x)-g_\Delta(y)|^2\\
			=& |\pi_\Delta(x)-\pi_\Delta(y)|^2+2\langle\pi_\Delta(x)-\pi_\Delta(y),f_\Delta(x)-f_\Delta(y)\rangle\Delta
			+|f_\Delta(x)-f_\Delta(y)|^2\Delta^2\\
			&+\mu\Delta|g_\Delta(x)-g_\Delta(y)|^2\\
			\leq& |\pi_\Delta(x)-\pi_\Delta(y)|^2+2K_3|\pi_\Delta(x)-\pi_\Delta(y)|^2\Delta
			+|f_\Delta(x)-f_\Delta(y)|^2\Delta^2\\
			=&(1+C\Delta)|\pi_\Delta(x)-\pi_\Delta(y)|^2+\Delta^2|f_\Delta(x)-f_\Delta(y)|^2:=(1+C\Delta)I_1+\Delta^2I_2^2.
		\end{flalign*}
		For $I_2$, it follows from Assumption \ref{as.1} and Remark \ref{Remark2} that
		\begin{flalign*}
			I_2\leq&K_1\Big(1+|\pi_\Delta(x)|^\alpha+|\pi_\Delta(y)|^\alpha+|\pi_\Delta(x)|^{-\beta}+|\pi_\Delta(y)|^{-\beta}\Big)|\pi_\Delta(x)-\pi_\Delta(y)|\\
			\leq&K_1\Big(1+2d^{\frac{\alpha}{2}}|\phi^{-1}(h(\Delta))|^\alpha+2d^{\frac{\alpha}{2}}|\phi^{-1}(h(\Delta))|^\beta\Big)|\pi_\Delta(x)-\pi_\Delta(y)|\\
			\leq&K_1\Big(1+C(\widetilde{\Delta}^\alpha+\widetilde{\Delta}^\beta)\Big)|\pi_\Delta(x)-\pi_\Delta(y)|.
		\end{flalign*}
		From Remark \ref{Remark2}, if $\alpha\geq\beta+1$, then we can derive $\widetilde{\Delta}^\alpha+\widetilde{\Delta}^\beta\leq 2\Delta^{-\frac{1}{2}}$. If $\alpha<\beta+1$, then $\widetilde{\Delta}^\alpha+\widetilde{\Delta}^\beta\leq 2\Delta^{-\frac{1}{2}}$ also can be derived. Consequently, 
		\begin{flalign*}
			&|\pi_\Delta(x)-\pi_\Delta(y)+(f_\Delta(x)-f_\Delta(y))\Delta|^2+\mu\Delta|g_\Delta(x)-g_\Delta(y)|^2\\
			\leq& (1+C\Delta)I_1+C\Delta^2(1+C\Delta^{-\frac{1}{2}})^2|\pi_\Delta(x)-\pi_\Delta(
			y)|^2\leq(1+C\Delta)I_1.
		\end{flalign*}
		According to $I_1$, we need prove $I_1\leq|x-y|^2$. For the symmetry of $x$ and $y$, we may as well suppose $|x|\leq|y|$. Using mathematical induction, we can derive the following estimates:
		\begin{flalign}\label{pipi}
			|\pi_\Delta(x)-\pi_\Delta(y)|\leq|x-y|.
		\end{flalign}
		When $d=1$, it is easily seen that \eqref{pipi} holds for $x,y\in(0,(\phi(h(\Delta)))^{-1})$, $((\phi(h(\Delta)))^{-1},\phi(h(\Delta)))$ or $(\phi(h(\Delta)),+\infty)$.
		\begin{flalign*}
			|\pi_\Delta(x)-\pi_\Delta(y)|= \left\{ 
			\begin{aligned}
				&|(\phi(h(\Delta)))^{-1}-y|\leq|x-y|,& &x\leq(\phi(h(\Delta)))^{-1}\leq y\leq \phi^{-1}(h(\Delta)), \\
				&|(\phi(h(\Delta)))^{-1}-\phi^{-1}(h(\Delta))|\leq|x-y|,& &x\leq(\phi(h(\Delta)))^{-1}\leq \phi^{-1}(h(\Delta))\leq y, \\
				&|x- \phi^{-1}(h(\Delta))|\leq|x-y|,& &(\phi(h(\Delta)))^{-1}\leq x\leq \phi^{-1}(h(\Delta))\leq y.\\
			\end{aligned}
			\right.
		\end{flalign*}
		Therefore, \eqref{pipi} is true for $d=1$. When $d=n$, we have \eqref{pipi}, i.e., 
		\begin{flalign*}
			|\pi_\Delta(x)-\pi_\Delta(y)|^2=&(\pi_\Delta(x_1)-\pi_\Delta(y_1))^2+\cdots+(\pi_\Delta(x_n)-\pi_\Delta(y_n))^2
			\leq (x_1-y_1)^2+\cdots+(x_n-y_n)^2
		\end{flalign*}
		for $x,y\in \RR^n_+$, where $x_j$ and $y_j$ denote the $j$-th element of $x$ and $y$, respectively. When $d=n+1$, 
		\begin{flalign*}
			|\pi_\Delta(x)-\pi_\Delta(y)|^2=& (\pi_\Delta(x_1)-\pi_\Delta(y_1))^2+\cdots+(\pi_\Delta(x_n)-\pi_\Delta(y_n))^2+(\pi_\Delta(x_{n+1})-\pi_\Delta(y_{n+1}))^2\\
			\leq&(x_1-y_1)^2+\cdots+(x_n-y_n)^2+(x_{n+1}-y_{n+1})^2.
		\end{flalign*}
		Therefore, \eqref{pipi} holds for $d=n+1$. Furthermore, we prove $I_1\leq|x-y|^2$. 
		
		\section{Proof of Lemma 4.7}\label{Appendix13}
		For an arbitrarily fixed $\Delta\in(0,1)$ and any $0\leq t\leq T$, there exists unique integer $k\geq 0$ such that $t_k\leq t <t_{k+1}$. By the H\"older inequality and Theorem 1.7.1 in \cite{M01}, it follows that
		\begin{flalign*}
			\E[|X(t)-X(t_k)|^p]=&\E\Big|\int_{t_k}^{t}f(X(s))\mathrm{d}s+\int_{t_k}^{t}g(X(s))\mathrm{d}B(s)\Big|^p\\
			\leq&C\Big(\Delta^{p-1}\int_{t_k}^{t}\E\big|f(X(s))\big|^p\mathrm{d}s+\Delta^\frac{p-2}{2}\int_{t_k}^{t}\E\big|g(X(s))\big|^p\mathrm{d}B(s)\Big).
		\end{flalign*}
		It follows from Remark \ref{Remark1} and Lemma \ref{Lm2.1} that
		\begin{flalign*}
			\E[|f(X(s))|^p]\leq K_1\E[(1+|X(s)|^{\alpha+1}+|X(s)|^{-\beta})^p]\leq C(1+\E[|X(s)|^{p(\alpha+1)}]+\E[|X(s)|^{-p\beta}])\leq C.
		\end{flalign*}
		Similarly, we can have $\E[|g(X(s))|^p]\leq C$. Consequently, the assertion \eqref{eq ED} can be obtained.

		\bibliographystyle{sn-bibliography}

	\end{document}